\documentclass[reqno,12pt]{amsart}
\usepackage{graphicx,amsfonts,amssymb,amsmath,amsthm,amscd,epsfig}
\usepackage[all]{xy}
\usepackage{color}
\usepackage{hyperref}
\usepackage{anysize} \marginsize{1.2in}{1.2in}{1in}{1in}
\hypersetup{colorlinks,linkcolor=magenta,citecolor=blue,filecolor=blue,urlcolor=blue}

\theoremstyle{plain} 
\newcounter{wow}
\newtheorem{theorem}    [wow]{Theorem}
\newtheorem{lemma}      [wow]{Lemma}
\newtheorem{cor}  [wow]{Corollary}
\newtheorem{prop}[wow]{Proposition}

\theoremstyle{definition}

\newtheorem*{definition*}           {Definition}

\theoremstyle{remark}
\newtheorem{remark}[wow]{Remark}

\definecolor{prpl}{rgb}{0.7, 0.0, 0.7}
\newenvironment{JT}{\noindent \color{prpl}{\bf JT:} \footnotesize}{}
\newenvironment{BB}{\noindent \color{blue}{\bf BB:} \footnotesize}{}

    \newcommand{\mysect}[1]{
  \subsection{}\hspace{-.5em}\textbf{#1.}
\vskip .5em}

\def\Sym{\operatorname{Sym}}

\def\id{\operatorname{id}}
\def\GL{\operatorname{GL}}
\def\PGL{\operatorname{PGL}}

\def\tr{\operatorname{tr}}

\def\End{\operatorname{End}}
\def\Aut{\operatorname{Aut}}
\def\SL{\operatorname{SL}}
\def\PSL{\operatorname{PSL}}

\def\mult{\operatorname{mult}}
\def\vol{\operatorname{vol}}

\def\gon{\operatorname{gon}}
\def\CUSP{\operatorname{SBC}}
\def\CM{\operatorname{CM}}

\def\Ram{\operatorname{Ram}}
\def\Span{\operatorname{Span}}

\def\Re{\operatorname{Re}}
\def\Im{\operatorname{Im}}
\def\Deg{\operatorname{Deg}}
\def\chone{\operatorname{c_1}}

\def\P{\mathbb{P}}
\def\Q{\mathbb{Q}}
\def\Z{\mathbb{Z}}
\def\C{\mathbb{C}}
\def\R{\mathbb{R}}
\def\G{\mathbb{G}}
\def\F{\mathbb{F}}
\def\H{\mathbb{H}}

\def\D{\mathbb{D}}

\renewcommand{\phi}{\varphi}
\renewcommand{\tilde}[1]{\widetilde{#1}}
\renewcommand{\bar}[1]{\overline{#1}}
\renewcommand{\mod}{\;\mathrm{mod}\;}
\newcommand{\textmod}{\mathrm{mod}}

\def\into{\rightarrow}

\newcommand{\mat}[4]{\begin{pmatrix}#1&#2\\#3&#4\end{pmatrix}}

 \title[$p$-torsion monodromy representations of elliptic curves]{$p$-torsion monodromy representations of elliptic curves over geometric function fields}
 \date{\today}
 \author{Benjamin Bakker}
 \address{B. Bakker:
 Institut f\"ur Mathematik, Humboldt-Universit\"at zu Berlin.
 }
 \email{benjamin.bakker@math.hu-berlin.de}
\author{Jacob Tsimerman}
\address{J. Tsimerman:
Mathematics Department, University of Toronto.}
\email{jacobt@math.toronto.edu}

\begin{document}
\begin{abstract}
Given a complex quasiprojective curve $B$ and a non-isotrivial family $\mathcal{E}$ of elliptic curves over $B$, the $p$-torsion $\mathcal{E}[p]$ yields a monodromy representation $\rho_\mathcal{E}[p]:\pi_1(B)\rightarrow \mathrm{GL}_2(\mathbb{F}_p)$.  We prove that if $\rho_{\mathcal E}[p]\cong \rho_{\mathcal E'}[p]$ then $\mathcal{E}$ and $\mathcal E'$ are isogenous, provided $p$ is larger than a constant depending only on the gonality of $B$.  This can be viewed as a function field analog of the Frey--Mazur conjecture, which states that an elliptic curve over $\mathbb{Q}$ is determined up to isogeny by its $p$-torsion Galois representation for $p> 17$.  The proof relies on hyperbolic geometry and is therefore only applicable in characteristic 0.
\end{abstract}
\maketitle
	
\section{Introduction}

The Frey--Mazur conjecture\footnote{See Fisher \cite{fisher} for a survey of the Frey--Mazur conjecture.  An explicit lower bound on $p$ was originally not specified, but Nicholas Billerey has found a counterexample for $p=17$ \cite{billerey}.}, originating in \cite{mazur}, states that for a prime $p>17$, an elliptic curve over $\Q$ is classified up to isogeny by its $p$-torsion, viewed as a Galois representation (or equivalently, as a finite flat group scheme).  A natural generalization of this conjecture asserts that over a fixed number field $K$ there is a uniform $M_K$ such that for primes $p>M_K$, elliptic curves over $K$ are classified up to isogeny by their $p$-torsion representation. Moreover, one can hope that $M_K$ can be taken to depend only on the degree of $K$. 

 Geometrically, there is a surface $Z(p)$ that parameterizes
triples $(E,E',\phi)$ consisting of a pair of elliptic curves $E,E'$ together with an isomorphism $\phi:E[p]\xrightarrow{\cong}E'[p]$ of their $p$-torsion.  This surface is endowed with natural Hecke divisors $H_m$ parametrizing points for which $\phi$ is induced by a cyclic isogeny of degree $m$. The Frey--Mazur conjecture is equivalent to the statement that for $p> 17$, all rational points of $Z(p)$ lie on one of these divisors\footnote{Note that by a theorem of Mazur, it is only necessary to consider $m\leq 163$.}, and in fact many arithmetic results are conjecturally related to the geometry of $Z(p)$ (see for example \cite{frey}).  For example, the surface $Z(p)$ is of general type for $p>11$ by work of Hermann \cite{hermann}, so the Bombieri--Lang conjecture implies that there are only finitely many rational points on the complement of all the rational and elliptic curves in $Z(p)$.  It is therefore natural to first study curves in $Z(p)$.

As our main result, we prove a function field analogue of the Frey--Mazur conjecture over the function field $K=k(B)$ of a complex curve $B$. Namely, we show that families of elliptic curves over $B$ are classified up to isogeny by the monodromy action on their $p$-torsion for any sufficiently large $p$. In fact, we prove the stronger statement that the constant $M_K$ depends only on the gonality of $B$. Recall that the gonality of an algebraic  curve
is the lowest degree finite map to $\P^1$, which is the analogue of the degree of a number field in the function field setting.  Precisely, we show:

\begin{theorem}\label{main}  Let $k$ be an algebraically closed field of characteristic 0.  For any $N>0$, there exists $M_N>0$ such that for any prime $p>M_N$ and any smooth quasiprojective curve $U$ of gonality $n<N$, non-isotrivial elliptic curves $\mathcal{E}$ over $U$ are classified up to isogeny by their $p$-torsion local system $\mathcal{E}[p]$.
\end{theorem}
Equivalently, choosing a basepoint $u\in U$, non-isotrivial elliptic curves $\mathcal{E}$ over $U$ are classified up to isogeny by the monodromy representation of the fundamental group $\pi_1(U,u)$ on the 2-dimensional $\F_p$ vector space $\mathcal{E}[p]_u$.

We can restate the above theorem in a way that seems more immediately analogous to the usual Frey--Mazur conjecture:
\begin{theorem}\label{main2}With $k, N, M_N$ as above and for any smooth projective curve $B$ of gonality $n<N$, non-isotrivial elliptic curves $E$ over the field $k(B)$ of rational functions on $B$ are classified up to isogeny by their $p$-torsion Galois representation provided $p>M_N$.
\end{theorem}

Note that since the gonality of modular curves gets large \cite{Abramovich,zograf}, for large enough $p$ the Galois representations are all surjective onto $\SL_2(\F_p)$\footnote{As we are working over complex curves, the Weil pairing is invariant under the monodromy action, so the representation lies in $\SL_2$ rather than $\GL_2$.} and hence geometrically irreducible, so we don't have to worry about semi-simplifying the representations.

Theorems \ref{main} and \ref{main2} follow from the geometric result:
\begin{theorem}[see Theorem \ref{gonality}]\label{geometric} For $k,N,M_N$ as above, every curve $B\subset Z(p)$ of gonality $n<N$ is a Hecke divisor provided $p>M_N$.  
\end{theorem}
Theorem \ref{geometric} proves a conjecture of Kani and Schanz \cite{kani} (and a related conjecture of Hermann \cite{hermann}) on the nonexistence of non-Hecke rational and elliptic curves in $Z(p)$ for large $p$ (\emph{cf.} Corollaries \ref{rational} and \ref{minimal}).  For the most part it is easier instead to study curves in the product $X(p)\times X(p)$, where $X(p)$ parameterizes elliptic curves together with an isomorphism $E[p]\cong (\Z/p\Z)^2$.  The surface $Z(p)$ is naturally the quotient of $X(p)\times X(p)$ only remembering the composition $E_1[p]\xrightarrow{\cong}(\Z/p\Z)^2 \xrightarrow{\cong}E_2[p]$.

The main idea of the proof of Theorem \ref{geometric} runs as follows:  given a curve $B$ of gonality $n<N$ in $Z(p)$, we first get a genus $0$ curve $\P^1\into \Sym^nZ(p)$ using the degree $n$ map $B\into \P^1$.  We then lift $\P^1$ to a curve $C\into(X(p)\times X(p))^n$ and estimate its genus using Riemann--Hurwitz in two different ways. On the one hand, we have a lower bound by the degree with respect to the canonical class $K_{(X(p)\times X(p))^n}$.  On the other hand, the ramification of $C\rightarrow \P^1$ is supported on the ramification points of the quotient map $$(X(p)\times X(p))^n\rightarrow\Sym^nZ(p)$$ and we show that the incidence of $C$ along this set is negligible compared to its degree for large $p$. This constitutes the heart of the paper.  We remark that to prove Theorem \ref{main} with a constant only depending on the genus of the base curve, one can bypass the more technically difficult statements involving the multidiagonals in Sections \ref{repulsionsect} through \ref{mult}.  In this case one cannot reduce to genus 0 curves, and therefore must include an additional argument, such as the topological one in \cite{BT}.

Our proof heavily relies on hyperbolic geometry and the fact that $(X(p)\times X(p))^n$ is uniformized by the $2n$-th power $\H^{2n}$ of the upper half-plane.  The strategy is to bound the multiplicity of curves $C$ along geodesic subvarieties of $(X(p)\times X(p))^n$ in terms of their volume in small tubular neighborhoods of those subvarieties.  A classical result of Federer  states that a curve in $\C^n$ passing through the origin must have volume in the ball $B(0,r)$ of radius $r$ at least equal to that of a coordinate axis.  This result was generalized heavily by Hwang and To \cite{hwangto1,hwangto2} to the case of arbitrary symmetric domains and higher-dimensional subvarieties.  For our needs, the theorems of Hwang and To are not quite sufficient, so we prove several analogues of these results, which may be interesting in their own right.

The bounds we obtain on the multiplicities of $C$ along geodesic subvarieties are better for large radius neighborhoods, but in order to bound the multiplicity along many such subvarieties simultaneously it is necessary to understand how these neighborhoods overlap.  We prove that special subvarieties tend to grow farther apart as $p$ gets large, and that these subvarieties only ``clump" together near higher-dimensional special subvarieties.  The proofs of these repulsion results are arithmetic in nature and fundamentally use the fact that the monodromy group of $X(p)$ over $X(1)$ is an algebraic group.

The proof of Theorem \ref{geometric} ultimately only uses the fact that elliptic curves are parametrized by a Shimura variety of dimension 1, and therefore we expect the same methods to prove an analogue of Theorem \ref{main} for abelian varieties parametrized by any Shimura curve---in fact the proof simplifies substantially due to the lack of cusps.  The case of abelian surfaces with quaternionic multiplication was treated in \cite{BT}\footnote{In \cite{BT} the authors only prove the weaker result that the map from isogeny classes to $p$-torsion representations is 2 to 1. This can be rectified by using the stronger repulsion statement found in Proposition \ref{repulsion}.} by the authors.
\mysect{Outline of the paper}\noindent
In Section \ref{modularsect} we recall background on the modular curves $Y(p)$, including the modular interpretation of the compactifications $X(p)$, and introduce the basic structures on $Z(p)$.  Our techniques require a uniformized metric on $(X(p)\times X(p))^n$, and in Section \ref{hyper} we study the uniformized metric on $X(p)$ in terms of the classical metric on $Y(p)$.  Section \ref{repulsionsect} establishes the repulsion of special subvarieties of the product $(X(p)\times X(p))^n$, and in Section \ref{volest} we provide some machinery in the style of Hwang and To estimating the volume of curves in small neighborhoods of these subvarieties.  Section \ref{mult} combines these results to provide estimates of the multiplicities of curves along special subvarieties, and in Section \ref{mainsect} we use this to estimate ramification and prove Theorem \ref{geometric}.

\mysect{Acknowledgements}\noindent  The authors benefited from many useful conversations with Fedor Bogomolov, Johan de Jong, Michael McQuillan, Allison Miller, and Peter Sarnak. The first named author was supported by NSF fellowship DMS-1103982 at the time of the writing of this paper.  Finally, we are greatly indebted to the referee for offering numerous suggestions for improving the clarity of the exposition as well as simplifications to the proofs in Section \ref{volest}.

\mysect{Notation and conventions}\noindent Throughout the paper we use the following notation regarding asymptotic growth:  for functions $f,g$ we write $f\gg g$ if there is a positive constant $C>0$ such that $f-Cg$ is a positive function; likewise for $\ll$.  If $f_t,g_t$ are functions depending on $t$, we write $f_t= O(g_t)$ if there is a positive constant $C>0$ such that $C|g_t|-|f_t|$ is positive for $t$ sufficiently large.  If the same is true for any $C>0$ we write $f_t=o(g_t)$. We also write $f_t=\omega(g_t)$ to mean $g_t=o(f_t)$. For us, the asymptotic parameter $t$ will always be the prime $p$.

Much of the paper will be concerned with the geometry of the hyperbolic plane, and both the upper half-plane model $\H$ and the Poincar\'e disk model $\D$ will prove convenient.  For computations we normalize the metric to have constant sectional curvature $-1$, so explicitly
\[h_\H=\frac{dz\otimes d\bar z}{(\Im z)^2}\indent\mathrm{and}\indent h_\D=4\cdot \frac{dz\otimes d\bar z}{(1-|z|^2)^2} .\]
We denote the associated distance functions by $d_\H$ and $d_\D$.  We also fix the implicit choice of normalization of the Kobayashi metric so that it coincides with the above metrics.  Note that the associated K\"ahler forms $\omega:=-\Im h$ are
\[\omega_\H=\frac{dz\wedge d\bar z}{2|\Im z|^2}\indent\mathrm{and}\indent \omega_\D=2\cdot \frac{dz\wedge d\bar z}{(1-|z|^2)^2}\]
where we canonically identify purely imaginary 2-forms with measures.  We define $d=\partial+\bar\partial$ as the total differential and $d^c=\frac{1}{4\pi}(\bar\partial-\partial)$.

For any hyperbolic curve $X$, we likewise obtain a metric $h_X$ and a form $\omega_X$ by descent along the universal cover.  We endow products $X^n$ with the Kobayashi metric as well, which is explicitly
\[d_{X^n}((x_1,\ldots,x_n),(y_1,\ldots,y_n))=\max_i d_X(x_i,y_i).\]
We also define the form $\omega_{X^n}$ as the sum of the pullbacks of $\omega_X$ along each projection.  Volumes will always be computed with respect to $\omega_{X^n}$.  If $X$ is compact we have
\begin{equation}\chone(K_X)=\frac{1}{2\pi}[\omega_X]\in H^{1,1}(X,\R(1)).\end{equation}  

To avoid any potential confusion involving the normalization as $p\into \infty$ we will phrase our results as often as possible in terms of manifestly normalization-independent quantities.  Thus, we define
\[a(r):=\mbox{area of the hyperbolic disk of radius $r$}\] 
For example, one can compute using the above normalization that
\[d_\D(0,z)=2\cdot \tanh^{-1}|z|\]
and therefore that
\[a(r)=4\pi\cdot \sinh^2(r/2).\]
We will also define, for any curve $C\subset X^n$,
\[\Deg (C):=K_{X^n}\cdot C\]
so that in the above normalization
\[\Deg (C)=\frac{1}{2\pi}\vol(C).\]
\section{Modular curves}
\label{modularsect}
\mysect{Basics on modular curves}
\noindent
For a prime number $p>3$ we let $Y(p)$ denote the coarse moduli scheme representing pairs $$(E,\phi:(\Z/p\Z)^2\xrightarrow{\cong} E[p])$$ of elliptic curves $E$ together with a \emph{projective} isomorphism $\phi$ from $(\Z/p\Z)^2$ to the $p$-torsion of $E$---that is, an isomorphism $\phi:(\Z/p\Z)^2\into E[p]$ defined up to scaling. We let $X(p)$ denote the standard smooth compactification of $Y(p)$; the added points $X(p)-Y(p)$ are referred to as cusps. $X(p)$ has 2 connected components, determined by the square class of the Weil pairing of $\langle\phi(e_1),\phi(e_2)\rangle$. We let $X(p)_{\epsilon}$ denote the corresponding connected component, where $\epsilon\in\F_p^{\times}/(\F_p^{\times})^2$. We shall consider these schemes exclusively over $\C$.

We recall that $\SL_2(\R)$ has a natural action on the upper half-plane $\H$ given by $$\left(\begin{smallmatrix} a & b\\ c & d\end{smallmatrix}\right)\cdot z = \frac{az+b}{cz+d}.$$

Letting $\Gamma(p):=\{\left(\begin{smallmatrix} a & b\\ c & d\end{smallmatrix}\right)\in\SL_2(\Z)\mid \left(\begin{smallmatrix} a & b\\ c & d\end{smallmatrix}\right)\equiv \left(\begin{smallmatrix} 1 & 0\\ 0 & 1\end{smallmatrix}\right) \mod p\}$, there is a natural isomorphism
$Y(p)_1\cong \Gamma(p)\backslash\H$ and for any $\epsilon\in\F_p^{\times}$ there exist (noncanonical) isomorphisms $Y(p)_1\cong Y(p)_{\epsilon}$. There is also a natural action of $\PGL_2(\F_p)$ on $Y(p)$ which permutes the two components, given by $$g(E,\phi):= (E,\phi\circ\tilde{g}^{-1})$$ where $\tilde{g}$ is any lift of $g$ to $\GL_2(\F_p)$.  The stabilizer of $Y(p)_\epsilon$ under this action is $\PSL_2(\F_p)$.
  Likewise there is an action of $\PGL_2(\F_p)$ on $X(p)$, and the quotient is canonically the compactified modular curve $X(1)$; call the quotient map $\pi: X(p)\into X(1)$. The ramification occurs at the cusps $X(p)- Y(p)$ and at the pre-images under $\pi$ of the points in $q_2,q_3\in X(1)$ representing the elliptic curves with CM by $\Z[i]$ and $\Z[e^{2\pi i/3}]$ respectively. Each of these three sets forms a single orbit under $\PGL_2(\F_p)$, and the ramification order at a point in the orbit is $p,2,$ and $3$ respectively.  

We remark that there is a natural anti-holomorphic involution on $Y(1)$ given by negating the complex structure of the elliptic curve, and this induces an involution on $Y(p)$ and $X(p)$ as well. We denote this by $z\rightarrow\bar{z}$.

\mysect{Modular interpretation of $n$-gons}
$X(1)$ has several interpretations as the coarse space of a moduli problem compactifying that of $Y(1)$; usually this is done by considering the cusp point as the pointed nodal cubic or ``1-gon," but when considering elliptic curves with $n$-torsion it is more naturally thought of as the $n$-gon.

\begin{definition*}  Let $C$ be the nodal cubic and let $C_n$ be the unique connected degree $n$ \'etale cover of $C$.  Geometrically $C_n$ is a cyclic chain of $\P^1$s, obtained from $\Z/n\Z\times \P^1$ by gluing $(k,\infty)$ to $(k+1,0)$ at a node.  An $n$-gon is $C_n$ together with a group structure on the smooth locus $C_n^{sm}$ such that the action of $C_n^{sm}$ on $C_n^{sm}$ extends to all of $C_n$.  This latter requirement implies that $C_n^{sm}$---which is $n$ copies of $C^{sm}$---is noncanonically $\Z/n\Z\times \G_m$ as a group scheme, where $\Z/n\Z$ acts by rotating the cycle of rational curves.
\end{definition*}
Note that the automorphism group of the $n$-gon is noncanonically $\Z/n\Z\ltimes\Z/2\Z$, the $\Z/2\Z$ coming from inversion on the smooth locus.  Indeed, choosing a smooth point $x\in C_n$ of order $n$ not in the identity component $(C_n^{sm})^0$ yields a group isomorphism
\[(C_n^{sm})^0\times \Z/n\Z \xrightarrow{\cong}C_n^{sm}: (t,m)\mapsto tx^m\]
and each map $(t,m)\mapsto (\zeta^m t,m)$ is a group automorphism of $(C_n^{sm})^0\times \Z/n\Z$ for any $n$-th root of unity $\zeta$, yielding the $\Z/n\Z$ of automorphisms.

The notion of $n$-gon is useful because one can be endowed with full level $n$ structure.  The $n$-torsion $C_n^{sm}[n]$ is canonically Cartier self-dual, and thus there is a natural Weil pairing.  We therefore define a level $n$ structure of the $n$-gon $C_n$ to be an isomorphism $C_n^{sm}[n]\cong (\Z/n\Z)^2$ up to scale.  For the rest of the paper, for a fixed prime $p$, by a generalized elliptic curve we will mean either an elliptic curve or a $p$-gon, so that we may speak of generalized elliptic curves with full level $p$ structure.

$X(p)$ can now be interpreted (for $p>3$) as the fine moduli space of generalized elliptic curves with full level $p$ structure.  The stack $\mathcal{X}(1)_p$ of generalized elliptic curves (in the above sense) without the level $p$ structure has coarse space $X(1)$ and compactifies the moduli problem $\mathcal{Y}(1)$; it differs from $\mathcal{X}(1)\cong \mathcal{X}(1)_1$ in that the cusp point in the $\mathcal{X}(1)_p$ moduli problem has an order $p$ stabilizer.

\begin{remark}
Throughout the above, we assume an algebraically closed (characteristic 0) base field, and thus blur the distinction between $\mu_n$ and $\Z/n\Z$.  Over a non-closed field, the $n$-torsion of $C_n^{sm}$ is noncanonically a Cartier self-dual extension of $\Z/n\Z$ by $\mu_n$, and the automorphism group of the $n$-gon is an extension of $\Z/2\Z$ by $\mu_n$.
\end{remark}

\mysect{The diagonal quotient surface}
We now introduce our primary object of study.

\begin{definition*} Let $Z(p)$ denote the coarse moduli scheme representing triples $$(E_1,E_2,\psi:E_1[p]\xrightarrow{\cong} E_2[p])$$ consisting of a pair of generalized elliptic curves $E_1,E_2$ together with a \emph{projective} isomorphism between their $p$-torsion.
\end{definition*}

Note that there is a natural morphism $X(p)\times X(p)\rightarrow Z(p)$ given by $$(E_1,\phi_1)\times(E_2,\phi_2)\rightarrow (E_1,E_2,\phi_2\circ\phi_1^{-1})$$ which identifies $Z(p)$ with
the quotient $\PGL_2(\F_p)\backslash X(p)\times X(p)$ by the diagonal action of $\PGL_2(\F_p)$.

The surface $Z(p)$ was introduced by Hermann \cite{hermann} and studied by Kani--Schanz \cite{kani} and Carlton \cite{carlton}.

\mysect{(anti-)Heegner CM points and singular bicusps}
\noindent

Note that $(x,y)\in X(p)\times X(p)$ is a ramification point of the quotient map $X(p)\times X(p)\into Z(p)$ exactly if $x$ and $y$ share a common stabilizer in $\PGL_2(\F_p)$. Thus, either $x,y$ are both cusps or both in $Y(p)$.

Suppose that $x,y\in Y(p)$ have a common stabilizer $g\in \PGL_2(\F_p)-\textbf{1}$, so that they both map to either $q_2$ or $q_3$ in $X(1)$. Now, since $g$ stabilizes $x,y$ there must
exist automorphisms $h_x,h_y$ of $E_x,E_y$ respectively and lifts $g_x,g_y$ of $g$ to $\GL_2(\F_p)$ such that $h_x\circ\phi_x = \phi_x\circ g_x^{-1}$ and $h_y\circ\phi_y=\phi_y\circ g_y^{-1}$. It is clear that neither of $h_x,h_y$ are $\pm1$.  By possibly negating $h_x,g_x$ we can ensure that $h_x,h_y$ have the same characteristic polynomial, either $t^2+1,t^2-t+1$ or $t^2+t+1$.
It follows that we can take $g_x=g_y$.

\begin{definition*}
Under the above setup, we say that $(x,y)\in Y(p)\times Y(p)$ is a {\it Heegner CM point} if the eigenvalues of $h_x$ acting on the tangent space $T_0E_x$ and $h_y$ acting on $T_0E_y$ are the same, and an {\it anti-Heegner CM point} otherwise, in which case they are complex-conjugates. We denote these sets by $\CM^+, \CM^-\subset X(p)\times X(p)$, respectively, and define $\CM=\CM^+\cup\CM^-$.
\end{definition*}

It is easy to see that $\PGL_2(\F_p)$ preserves each of the sets $\CM^+$ and $\CM^-$.  Note that $(x,y)$ is a Heegner CM point if and only if $(x,\bar y)$ is an anti-Heegner CM point.

In the terminology of Kani and Schanz \cite{kani}, a \emph{Heegner CM point} of $Z(p)$ is a point of the form $(E,E,\psi|_{E[p]})$ for an elliptic curve $E$ with $\Aut(E)\neq\pm \id$ and $\psi\in\End(E)$ of degree coprime to $p$, whereas an \emph{anti-Heegner CM point} is one of the form $( E, \bar E,\tau\circ\psi|_{E[p]})$ for such $E$ and $\psi$, where $\tau: E\into \bar E$ is complex conjugation.  Thus the (anti-)Heegner CM points of $X(p)\times X(p)$ lie over (anti-)Heegner CM points $(E,E,\phi)$ of $Z(p)$.

The only other ramification points of the map $X(p)\times X(p)\into Z(p)$ are those whose coordinates are cusps with a common stabilizer.  We refer to these as \emph{bicusps}, and make the

\begin{definition*} A point $(x,y)\in X(p)\times X(p)$ is called a {\it singular bicusp} if $x,y$ are both cusps and they share a stabilizer in $\PGL_2(\F_p)$. We denote the set of them by $\CUSP\subset X(p)\times X(p)$.
\end{definition*}
\mysect{Hecke operators}\label{hecke}

Recall that if we have an integer $n$ relatively prime to $p$, we can define a Hecke correspondence $T_n\subset X(p)\times X(p)$ between $X(p)$ and itself as the closure of:
$$\{\left((E_1,\phi_1),(E_2,\phi_2)\right)\mid \exists\textrm{ cyclic isogeny }\psi:E_1\rightarrow E_2, \deg\psi=n, \psi\circ\phi_1=\phi_2\}.$$

The modular interpretation can be extended to $p$-gons with full level-structure as follows. First, consider the map  $\phi_n:C_{pn}\rightarrow C_p$ which on the smooth part is just
$$\phi_n(x,a) = (x^n,a), \indent(x,a)\in\G_m\times \Z/n\Z.$$
This induces an isomorphism on $p$-torsion, which we denote by $\phi_n[p]$.  Next, consider the set $G_n$ of all subgroups $G\subset C_{pn}^{sm}$ which are cyclic of order $n$ and whose intersection with the identity component is a single point. For each such $G\in G_n$ we get a map $\psi_G:C_{pn}^{sm}\rightarrow C_p^{sm}$ by quotienting out by $G$, which completes to a map from $C_{pn}$ to $C_p$. Finally, consider the map $\xi_d:C_p\rightarrow C_p$ which is $x\rightarrow x^d$ on the smooth part of the identity component, so that
$$\xi_d(x,a)=(x^d,a), (x,a)\in\G_m\times\Z/n\Z.$$

Now, to a point $(C_p,f:(\Z/p\Z)^2\cong C_p^{sm}[p])$ the $n$th Hecke operator associates the set $$\bigcup_{m|n}\bigcup_{G\in G_m}(C_p, \xi_{\frac{n}{m}}\circ\psi_G\circ\phi_m[p]^{-1}\circ f).$$

It will be important for us that $T_n$ is diagonally $\PGL_2(\F_p)$-invariant.  The two projection maps $\alpha,\beta:T_n\into X(p)$ both have degree $\deg(T_n)=\sigma_1(n)=\sum_{d|n}d$, and we denote by $\mu,\nu:X(p)\times T_n\into X(p)\times X(p)$ the maps $\mu=\id\times\alpha$ and $\nu=\id\times\beta$.  $\mu,\nu$ yield a correspondence from $X(p)\times X(p)$ to itself, and we denote by $T_m^*=\mu_*\nu^*$ the pullback of divisors along this correspondence.
\section{Hyperbolic properties of $X(p)$}\label{hyper}The modular curve $Y(p)$ carries its natural uniformized metric $h_{Y(p)}$ of constant sectional curvature $-1$.  The compactification $X(p)$ is a hyperbolic curve (for $p>5$) in and of itself, and thus also carries a uniformized metric $h_{X(p)}$.    Though the classical metric $h_{Y(p)}$ is well-understood, it's singularities at the cusps make it unamenable to the techniques of Hwang and To.  The purpose of this section is to study the properties of $h_{X(p)}$ by comparing it with $h_{Y(p)}$.

The orbifold coarse space of the stack $\mathcal{X}(1)_p$ of generalized elliptic curves is topologically a sphere with a point $q_2$ of order 2, and point $q_3$ of order 3, and a point $q_p$ of order $p$.  Let $X(1)_p$ be the (orbifold) Riemann surface associated to the tiling of the upper half-plane by a $(2,3,p)$ triangle in the same way that $Y(1)$ is associated to the tiling by the $(2,3,\infty)$ triangle (see for example \cite{beardon} for more on triangle groups).  Let
\[\Gamma(2,3,p)=\langle\sigma_2,\sigma_3,\sigma_p|\sigma_2^2=\sigma_3^3=\sigma_p^p=\sigma_2\sigma_3\sigma_p=1\rangle\cong\pi_1(X(1)_p)\]
be the $(2,3,p)$ triangle group.  The unique biholomorphism from the usual fundamental domain of $Y(1)$ to a $(2,3,p)$ triangle yields by Schwarz reflection the holomorphic embedding $i_p:Y(1)\into X(1)_p$, by which we identify $X(1)_p$ with the orbifold coarse space of $\mathcal{X}(1)_p$.  We let
$$\gamma_p:=i_{p*}:\pi_1(Y(1))\into\pi_1(X(1)_p).$$
After the usual identification $\pi_1(Y(1))=\PSL_2\Z\cong \Gamma(2,3,\infty)$, $\gamma_p$ is the obvious map $\Gamma(2,3,\infty)\into \Gamma(2,3,p)$.  For each connected component $X(p)_\epsilon$, the forgetful map $X(p)_\epsilon\into X(1)_p$ (of coarse spaces) is identified with the \'etale cover associated to the image $\Xi(p)$ of $\Gamma(p)$ under $\gamma_p$, and the metric $h_{X(p)}$ on $X(p)$ is the pullback of the uniformized metric on $X(1)_p$ coming from the above tiling (on each component). We define
$$G_0(p):=\Gamma(2,3,p)/\Xi(p)\cong \PSL_2(\F_p)$$
to be the Galois group of $X(p)_1$ over $X(1)_p$, and $G(p)\cong \PGL_2(\F_p)$ the canonical $\Z/2\Z$ extension acting on $X(p)$ such that the quotient is $X(1)_p$.
\mysect{Injectivity radii}
Recall that for a compact Riemann surface $X$ endowed with its metric $h_X$, the injectivity radius $\rho_{X}(x)$ at a point $x\in X$ is the largest radius for which the exponential map at $x$ is a diffeomeorphism.  It is equal to half the length of the smallest closed geodesic through $x$.  The injectivity radius $\rho_{X}$ is the infimum of $\rho_{X}(x)$ over all $x\in X$, or equivalently half the length of the shortest closed geodesic in $X$.  If we now allow $X$ to have cusps, we define $\rho_{X}$ to be the infimum of the lengths of closed geodesics with respect to $h_X$ that are homotopically nontrivial in the smooth compactification $X'$.

It was first observed by Buser--Sarnak in \cite{BuSa} that the injectivity radius of $Y(p)$ with respect to the metric $h_{Y(p)}$ grows. For the convenience of the reader, we recall the proof:

\begin{lemma}\label{injrad}$\rho_{Y(p)}= 2\log p+O(1)$.
\end{lemma}
\begin{proof}

The kernel of $\gamma_p$ is the group generated by the unipotents in $\Gamma(p)$. Thus, a homotopically nontrivial closed geodesic through $x\in Y(p)$ lifts to the unique geodesic arc between two lifts $z,\gamma z\in\H$ for some semisimple $\gamma\in \Gamma(p)$, so that $d(z,\gamma z)=d_\H(Az,aAz)$, where $A\in \SL_2\R$ is the diagonalizing matrix and $\sqrt a+\frac{1}{\sqrt a}=|\tr\gamma|$. In particular, using the formula for distance in the upper half-plane, this means
\begin{align*}
\min_{z}d_\H(z,\gamma z)&=\min_z d_\H(z,az)\notag\\
&=\min_z 2\tanh^{-1}\left|\frac{z-az}{z-a\bar z}\right|\notag\\
&=2\tanh^{-1}\left(\frac{a-1}{a+1}\right)\\
&=\log a\\
&\geq 2\log |\tr\gamma|+O(1).
\end{align*}
Note that for any $\gamma\in \Gamma(p)$, $\tr(\gamma) \equiv 2\mod{p^2}$.  Furthermore, the only semisimple element $\gamma$ with trace 2 is the identity, and thus the minimal value of $|\tr \gamma|$ over all semisimple $\textbf{1}\neq \gamma\in \Gamma(p)$ is $p^2-2$. Moreover, this bound can be achieved by taking $\gamma=\left(\begin{smallmatrix}1-p^2&p\\-p&1\end{smallmatrix}\right)$.  The result then follows.
\end{proof}

\mysect{Comparison of $h_{Y(p)}$ and $h_{X(p)}$}

This subsection will show that the metrics $h_{X(p)}$ and $h_{Y(p)}$ can be compared away from the cusps, and that the metric $h_{X(p)}$ on $X(p)$ also has growing injectivity radius.  After writing this paper, the authors were informed by Peter Sarnak that a comparison of these metrics by ``softer" curvature methods can be found in Brooks \cite{brooks}.  Much of this subsection can almost certainly be rederived using his methods.  We prefer the explicit analysis below, but caution the reader that the dependence on ``hard" results like the uniformization theorem is by choice rather than necessity.

Let $T_{2,3,p}\subset \H$ denote the $(2,3,p)$ triangle with vertices at $i, iy_p, e^{i\theta_p}$, with a right angle at $i$, $0<\theta_p<\pi/2$, and an angle of $\frac{\pi}{3}$ at $e^{i\theta_p}$. We define $\Delta_{2,3,p}$ to be the fundamental domain (see Figure \ref{fig:fundomain}) of the corresponding action of $\Gamma(2,3,p)$ on $\H$ given by the union of $T_{2,3,p}$ with its reflection through the imaginary axis.  Likewise, $\Delta_{2,3,\infty}$ is the usual fundamental domain for $Y(1)$. By the second hyperbolic law of cosines, we compute $\cos(\pi/3)=\sin(\pi/p)\cosh(\log y_p)$ from which we conclude that $y_p=p/\pi+O(p^{-1})$. Similarly, $\theta_p=\pi/3 +O(p^{-1})$.

 \begin{figure}[htbp]
\begin{center}

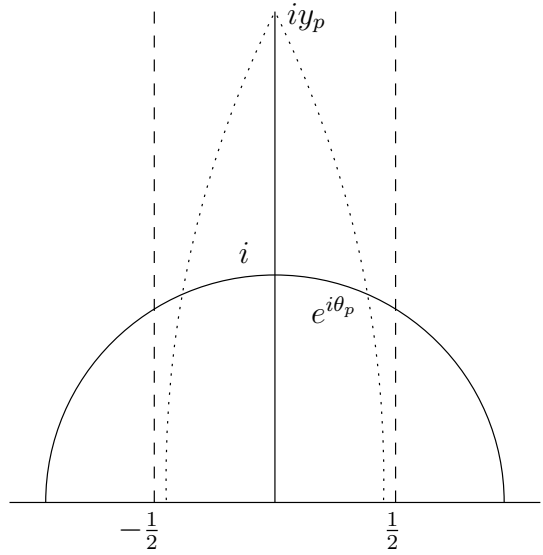
\caption{ \small Fundamental domains of $Y(1)$ and $X(1)_p$.  $\Delta_{2,3,\infty}$ is bordered by the arc of the unit circle together with the two dashed lines, while $\Delta_{2,3,p}$ is bordered by the arc of the unit circle together with the two dotted lines.}
\label{fig:fundomain}
\end{center}
\end{figure}

  Henceforth we we will implicitly think of $\Gamma(2,3,p)$ as embedded in $\PSL_2\R$ via the tiling by $T_{2,3,p}$ and its reflection.  $X(p)$ is then tiled by the set $\Sigma$ of the images of these triangles, so that $G(p)$ acts on $\Sigma$ with 2 orbits. It is easy to see (from Figure \ref{fig:fundomain}, for instance) the

\begin{lemma}\label{disksep}
For any cusp $c\in X(p)$, the disk of radius $\log p-O(1)$ centered at $c$ intersects only those triangles in $\Sigma$ with $c$ as a vertex.
As a consequence, if $c,c'$ are two cusps, then $d_{X(p)}(c,c')>2\log p-O(1)$.
\end{lemma}

Note that the Kobayashi metrics on $Y(p)$ and $X(p)$ are $h_{Y(p)}$ and $h_{X(p)}$, respectively.  Any holomorphic map is distance decreasing with respect to the Kobayashi metric, so from the embedding $j_p:Y(p)\into X(p)$ we immediately have $h_{X(p)}|_{Y(p)}\leq h_{Y(p)}$.  In fact, the two metrics are close far from the cusp:

\begin{prop}\label{metcom1} For $x\in X(p)$, let $d_{cusp}(x)=\inf_c d_{X(p)}(x,c)$ be the minimum distance to a cusp.  Then
\[\tanh^2(d_{cusp}/2)h_{Y(p)}\leq h_{X(p)}|_{Y(p)}\leq h_{Y(p)}.\]
\end{prop}
\begin{proof}The right hand inequality was addressed above.  The lift $\tilde{j}_p:\H\into\H$ of the inclusion $j_p$ to the universal covers is invertible at any non-cusp point $z\in\H$ on a hyperbolic disk $B(z,d_{cusp}(x))$ where $x\in X(p)$ is the image of $z$, so again by the distance-decreasing property $h_{Y(p)}$ bounds from below the Kobayashi metric of $B(z,d_{cusp}(x))$.   By scaling the hyperbolic disk $B(0,r)\subset \D$ by a factor of $1/\tanh(r/2)$ we see that the Kobayashi metric of $B(0,r)$ at $0$ is equal to $1/\tanh^2(r/2)$ times the hyperbolic metric, whence the claim.


\end{proof}
We can then conclude that the injectivity radius of $X(p)$ is close to that of $Y(p)$:
\begin{cor}\label{biginj}
The injectivity radius $\rho_{X(p)}=2\log p+O(1)$.
\end{cor}

\begin{proof}

The upper bound follows by Lemma \ref{injrad} combined with the fact that $h_{X(p)}\leq h_{Y(p)}$. For the lower bound,
suppose $\gamma$ is a minimal length geodesic loop in $X(p)$, of length $\ell_{X(p)}(\gamma)$.  Note that $\gamma$ can be within $(\log p)/2$ of at most one cusp, as $\ell_{X(p)}(\gamma)\leq 2\log p+O(1)$.
Thus, as $\gamma$ is a geodesic, the length of $\gamma$ in the $h_{X(p)}$ metric within a distance $d<(\log p)/2$ of any cusp is at most $2d$. Consider a new loop $\gamma'$ which is equal to $\gamma$, except that the stretch of $\gamma$ between when it first enters the disk of radius $1$ around the cusp and when it last exits that disk is replaced by an arc on   the boundary of the disk. Thus, by Proposition \ref{metcom1}
\begin{align*}
2\rho_{Y(p)}&\leq\ell_{Y(p)}(\gamma')\\
&\leq \ell_{X(p)}(\gamma') + \left(\ell_{Y(p)}(\gamma')-\ell_{X(p)}(\gamma')\right)\\
&\leq O(1) + \ell_{X(p)}(\gamma)+\int_{x=1}^{\infty}2\left(\frac{1}{\tanh(x/2)}-1\right)dx \\
&=2 \rho_{X(p)}+O(1).\\
\end{align*}
The claim follows by Lemma \ref{injrad}.
\end{proof}

The Hecke correspondences $T_m$ are isometric with respect to $h_{Y(p)}$, and we would now like to show that they are approximately isometric with respect to $h_{X(p)}$ in an appropriate sense.  Note that the usual uniformization $\H\into Y(p)_1$ factors through $e^{2\pi iz/p}:\H\into\D^*$, and furthermore that the map $\D^*\into Y(p)_1$ is injective on the image of $\H_{>y}:=\{z\in \H\mid \Im z>y\}$ for $y>1/p$.  Indeed for any $M=\mat{a}{b}{c}{d}\in \Gamma(p)$ with $c\neq 0$,
\[\Im M\cdot z=\frac{\Im z}{|cz+d|^2}\leq \frac{1}{p^2\cdot \Im z}.\]
Thus for $y>1/p$ we have an embedded disk $\D_y(c)\subset X(p)$ of Euclidean radius $e^{-2\pi y/p}$ around each cusp $c$ of $X(p)$, and by the distance decreasing property of the Kobayashi metric, 
\begin{equation}\label{comp1}\D_y(c)\subset B_{X(p)}(c,2\tanh^{-1}(e^{-2\pi y/p})).\end{equation}  On the other hand, by pulling back from $X(1)_p$, we see from the above that
\[B_{X(p)}(c,\log p+O(1))\subset \D_1(c).\]
By Proposition \ref{metcom1}, we can then conclude
\begin{align}
B_{X(p)}(c,R)&\subset \left\{x\mid d_{X(p)}(x,\D_1(c))<R-\log p+O(1)\right\} \notag\\
&\subset \D_{O(pe^{-R})}(c)\label{comp2}\end{align}
provided $\log p+O(1)<R<2\log p-O(1)$.  We are now in a position to prove the 
  \begin{lemma}\label{heckepullback}
  For any $R,m>0$ such that $R>\log p+O(1)$ and $R+\log m<2\log p-O(1)$, and any  cusp $c\in X(p)$, if $T_m^*c = \cup_i c_i$ then $$T_m^*B_{X(p)}(c,R) \subset\bigcup_i B_{X(p)}(c_i,R + \log m + O(1)).$$
  \end{lemma}
  \begin{proof}
Note that $T_m$ on the upper half-plane satisfies $\Im(T_m^*z)\geq \Im(z)/m$, and therefore for each cusp $c\in X(p)$ we have $T^*_m(\D_y(c))\subset \D_{y/m}(c)$ provided $y>m/p$.  The claim then follows from \eqref{comp1} and \eqref{comp2} above.
  \end{proof}

\mysect{Heights}

Fix now the uniformization $\H\into X(1)_p$ coming from the tiling by $\Delta_{2,3,p}$ and denote by $H=\langle \sigma_p\rangle\subset \Gamma(2,3,p)$ the stabilizer of $iy_p$, where $\sigma_p$ is rotation by $\frac{2\pi}{p}$ around $iy_p$.  We also fix a uniformization $\H\into X(p)_1$ associated to a subgroup $\Xi(p)\subset \Gamma(2,3,p)$.  For convenience, we define the usual notion of (big) height on $\SL_2(\Z)$:

\begin{definition*}  For $M\in \SL_2(\Z)$ denote by $h(M)$ the maximum absolute value of its entries.
\end{definition*}
The following two lemmas roughly say that for certain $x\in X(p)_1$, the distance between $x$ and a translate $M\cdot x$ is controlled by the height of $M$.
\begin{lemma} \label{c2dist}

Fix $\iota\in X(p)_1$ to be the image of $i\in\H$, and take $\delta<1/2$.  If $d_{X(p)}(\iota,\gamma\cdot \iota) < \delta\rho_{X(p)}$ for $\gamma\in G(p)$, then there exists a matrix $M\in \SL_2(\Z)$ with $h(M)=O(p^{4\delta})$ such that
$\gamma=\gamma_p(M)\mod \Xi(p)$.
\end{lemma}

\begin{proof}
Let $R=\delta\rho_{X(p)}$.  By Proposition \ref{metcom1} and Corollary \ref{biginj} it suffices to show in $\H$ that if $d_\H(i,M\cdot i) = R$ for $M\in \SL_2(\Z)$ then the elements of $M$ are $O(e^{2R})$.

Let $$M=\begin{pmatrix} a & b\\c  & d\end{pmatrix},$$ so that \begin{equation}M\cdot i = \frac{ai+b}{ci+d} = \frac{(bd+ac) + i}{c^2+d^2}\label{real}.\end{equation} It follows by looking at the imaginary part of $M\cdot i$ that
$$c^2+d^2\leq e^R$$ and thus each of $c,d$ is at most $e^{R/2}$. Using the distance formula on the upper half-plane we have

$$2\Im(M\cdot i)(\cosh(R)  - 1) = (\Re M\cdot i)^2 + (-1+\Im (M\cdot i))^2$$ and thus $$(\Re M\cdot i)^2\leq 2\Im(M\cdot i)e^R$$
which from equation \eqref{real} gives $$\frac{bd+ac}{c^2+d^2}\leq \sqrt{2\Im(M\cdot i)e^R}$$ and $$bd+ac\leq 2e^{\frac{3R}{2}}.$$ Now, using $ad-bc=1$ we get $$a(c^2+d^2)\leq  2ce^{\frac{3R}{2}} + d$$ so that
$$a\leq 2e^{2R}$$ and similarly for $b$.

\end{proof}
Fixing $\varpi\in X(p)_1$ to be the image of $e^{i\theta_p}\in\H$, an analogous result holds for $\gamma\in G(p)$ for which $\gamma\cdot \varpi$ is close to $\varpi$.  A slightly different statement is required for the cusp:
\begin{lemma} \label{cuspdist}

Fix $c_0\in X(p)_1$ to be the image of $iy_p\in\H$ and take $\delta<1/2$.
\begin{enumerate}
\item[(a)] For any cusp $c\in X(p)_1$  and any lift $z\in\H$ of $c$ with $d(iy_p,z)\leq (1+\delta)\rho_{X(p)}$, there exists $M\in\SL_2(\Z)$ with $h(M)=O(p^{4\delta})$ such that $z=\sigma\gamma_p(M)\cdot iy_p$ for some $\sigma\in H$.
\item[(b)]For all $\gamma\in G(p)$, if $d_{X(p)}(c_0,\gamma \cdot c_0) \leq (1+\delta)\rho_{X(p)}$ then there exists $M\in\SL_2 \Z$ with $h(M) = O(p^{4\delta})$ such that
$\gamma\in H\cdot \gamma_p(M)\cdot H\mod \Xi(p)$.
\end{enumerate}

\end{lemma}

\begin{proof}
By Lemma \ref{disksep} and Corollary \ref{biginj} we have $d(iy_p,z) > 2\log p -O(1)$, and so $$B(iy_p,(1+2\delta) \log p )\cap B(z, (1+2\delta)\log p)$$ is within a distance of $2\delta\log p$ of the boundary of $B(iy_p,\log p)$ as well as that of $B(z,\log p)$.
Thus, it is within $2\delta\log p +O(1)$ of a point $u\in\H$ which projects to $q_2$ in $X(1)_p$ and is a vertex of one of the $(2,3,p)$ tiles having $iy_p$ as a vertex, so that
$u=\sigma\cdot i$ for some $\sigma\in\langle \sigma_p\rangle$,
 It follows by Lemma \ref{c2dist} that there is a matrix $M\in \SL_2(\Z)$ with $h(M)= O(p^{4\delta})$ such that
$\gamma_p(M)\cdot i = \sigma^{-1}\cdot u$, and by possibly pre-composing $\gamma_p(M)$ with $\sigma_2$ we get that $\sigma\gamma_p(M) \cdot iy_p = z$, thus proving (a).

Part (b) easily follows from part (a) after choosing a lift $z\in \H$ of $\gamma\cdot c_0$ with $d(c_0,\gamma\cdot c_0)=d(iy_p,z)$.

\end{proof}\section{Repulsion Results}
\label{repulsionsect}
The volume estimates from Section 5 will allow us to bound the multiplicity of curves $C$ in $X(p)\times X(p)$ along special points in terms of their volume near those points.  As $p$ grows, the points tend to spread out further, and we can find neighborhoods of large radius around them that tend to be disjoint.  The total volume, and therefore the total multiplicity, can therefore be bounded by the total volume of the curve.  This argument fails when such neighborhoods overlap many times, but luckily such overlaps only occur close to higher-dimensional special subvarieties, which themselves repel one another.

Throughout this section, we solely consider the metric $h_{X(p)}$ from Section \ref{hyper} on $X(p)$, and suppress its mention from the notation.  Thus, the distance $d_{X(p)}(x,x')$ between $x,x'\in X(p)$ with respect to $h_{X(p)}$ will be denoted $d(x,x')$, and the ball around $x$ by $B(x,R)$.  We use the same notation for distance and balls in $\H$ as in $X(p)$, and rely on context to distinguish between the two.  For the most part we state our results in a normalization-independent way in terms of the injectivity radius $\rho_{X(p)}$.

A final advisory:  we use the phrase ``for all sufficiently small $\delta>0$ and all sufficiently large $p$" to mean that for each sufficiently small $\delta>0$ there is a constant $F_\delta>0$ such that the statement holds for $p>F_\delta$.

\mysect{Repulsion of cusps}

\begin{prop}\label{repulsioncusp}

For all sufficiently small $\delta>0$ and all sufficiently large $p$:

\begin{enumerate}
\item[(a)] For any distinct cusps $c_0,c\in X(p)$, and any pre-image $z_0\in\H$ of $c_0$, there is at most one pre-image $z\in\H$ of $c$ such that $d(z_0,z)\leq (1+\delta)\rho_{X(p)}$.

\item[(b)] For any distinct cusps $c_0,c\in X(p)$, $$B(c_0,(1/2+\delta)\rho_{X(p)} )\cap B(c, (1/2+\delta)\rho_{X(p)})$$ is contained in a ball of radius $\delta\rho_{X(p)} + O(1)$.

\item[(c)] For any $x\in X(p)$ there are at most $O(p^{24\delta})$ cusps $c\in X(p)$ within a distance $(1/2+\delta)\rho_{X(p)} $ of $x$.

\end{enumerate}

\end{prop}

\begin{proof}

To prove part (a), first note that because the automorphisms of $X(p)$ act transitively on the cusps, it suffices to take $z_0=iy_p$.

Let $z,z'\in\H$ be pre-images of cusps $c,c'$ with $d(iy_p,z)\leq (1+\delta)\rho_{X(p)}$, and likewise for $z'$. We'll first show that $c\neq c'$.  By Lemma \ref{cuspdist} part (a), there are matrices $M,M'\in \SL_2(\Z)$ with $h(M),h(M')= O(p^{4\delta})$ and elements
$\sigma,\sigma'\in\langle\sigma_p\rangle$ such that $\sigma\gamma_p(M)\cdot iy_p=z $ and $\sigma'\gamma_p(M') \cdot iy_p = z'$.

 Since $z,z'$ are translates under $\Xi(p)$, then $\gamma_p(MM'^{-1})\in \Xi(p)\langle \sigma_p \rangle$. Since the kernel of $\gamma_p$ is contained in $\Gamma(p)$
 we have that $MM'^{-1}\in\Gamma(p)U$, where $U$ is the upper triangular group. In other words, the reduction of $MM'^{-1}$ modulo $p$ is upper triangular. However, as the entries of $MM'^{-1}$ are  $O(p^{8\delta})$, it follows for $\delta<\frac14$ that for  large enough $p$ we must have $MM'^{-1}\in U$, which implies that $\gamma_p(MM'^{-1})\in \langle \sigma_p \rangle$, and
 $$\sigma\gamma_p(MM'^{-1})\sigma'^{-1}\in \Xi(p)\cap \langle \sigma_p\rangle = 1.$$ Thus we must have $z=z'$. This proves (a).

To prove part (b), after fixing a lift $z_0$ of $c_0$, $c$ has at most one lift within distance $(1+\delta)\rho_{X(p)}$ of $z_0$ by part (a). The claim then follows from the following lemma:

\begin{lemma}

For all $R>0$ there exists $M=R+O(1)$ such that the following is true:
Let $z,z'\in\H$ such that $d(z,z') = 2D$. Then $B(z,D+R)\cap B(z',D+R)$ is contained in a ball of radius $M$ centered at the midpoint of $z$ and $z'$.

\end{lemma}

\begin{figure}[htbp]
\begin{center}

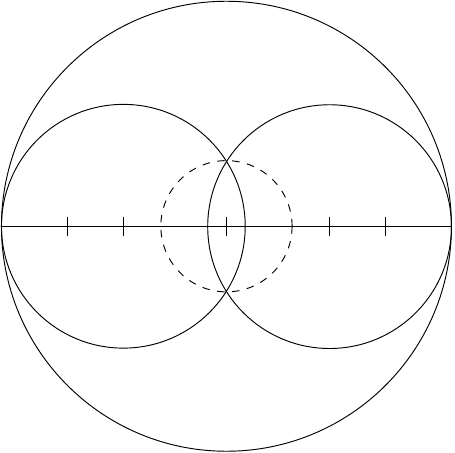
\caption{}
\label{fig:disk}
\end{center}
\end{figure}
\begin{proof}
We work in the Poincar\'e disk model. It suffices to consider $z=r, z'=-r$ for an appropriate real point $0<r<1$. Let $s=\tanh(R/2)$ so that $s$ is distance $R$ to $0$.
 Then $B(z,D+R)$ is contained in the \emph{Euclidean} disk $B_e(\frac{1-s}{2},\frac{1+s}{2})$, as both are bounded by (Euclidean) circles with centers on the $x$-axis and the same containment of the diameters along the $x$-axis clearly holds.  Likewise $B(z',D+R)$ is contained in the Euclidean disk $B_e(\frac{s-1}{2},\frac{1+s}{2})$. These 2 disks intersect in a convex figure which is contained in the Euclidean disk $B_e(0,\sqrt{s})$ (see Figure \ref{fig:disk}), or equivalently the hyperbolic disk $B(0,2\tanh^{-1}(\sqrt{s}))$.

 Now, $\tanh(x) = 1-2e^{-2x}+O(e^{-4x})$ for $x>0$ and $\tanh^{-1}(1-\delta) = -\ln{\delta}/2 + O(1)$ for $\delta<1/2$. Thus
\begin{align*}
2\tanh^{-1}(\sqrt{s}) &= 2\tanh^{-1}\left(\sqrt{1-2e^{-R}+O(e^{-2R})}\right)\\
&= 2\tanh^{-1}(1-e^{-R}+ O(e^{-2R}))\\
&=2(R/2 + O(1))\\
&=R + O(1).
\end{align*}
\end{proof}

To prove (c), we note that by (a) we may assume $x$ is within $\delta\rho_{X(p)} + O(1)$ of a point mapping to $q_2$. Thus, it suffices to prove that $\iota$ is within $(1/2+2\delta)\rho_{X(p)}+O(1)$ of at most $O(p^{24\delta})$ cusps. If $c$ is such a cusp, then $\iota$ must be within $2\delta \rho_{X(p)}+O(1)$ of a point $x$ mapping to $q_2$ neighboring $c$. By Lemma \ref{c2dist} there are at most
$O(p^{24\delta})$ such points $x$. This completes the proof.

\end{proof}

\mysect{Repulsion of CM points}\noindent\label{repel:CM}

We refer to the Kobayashi metric $d_{X(p)\times X(p)}$ simply as $d$.  A ball in this metric is a product of balls on each factor:
\[B(\xi,R)=B(x,R)\times B(y,R)\]

\begin{prop}\label{repulsion}
Let $\xi,\xi'\in \CM$ be distinct CM points in $X(p)\times X(p)$ with the same projections to $X(1)_p\times X(1)_p$ such that $B(\xi,\delta\rho_{X(p)})\cap B(\xi',\delta\rho_{X(p)} )\neq \varnothing$ and either both $\xi$ and $\xi'$ are in $\CM^+$ or both are in $\CM^-$.  Then for all sufficiently small $\delta>0$ and all sufficiently large $p$, $\xi$ and $\xi'$ lie on some Hecke divisor $T_m$ with $m =p^{O(\delta)}$.
\end{prop}

\begin{proof}
Let $\xi=(x,y)$ and $\xi'=(x',y')$.
 Diagonally acting by $G(p)$ does not affect the statement of the proposition. We thus assume without loss of generality that either $x=\iota$ or $x=\varpi$.  Take $t_0=\left(\begin{smallmatrix} 0 & 1\\-1 & 0\end{smallmatrix}\right)\in\SL_2(\Z)$ or $t_0=\left(\begin{smallmatrix} 0 & 1\\-1 & -1\end{smallmatrix}\right)\in\SL_2(\Z)$, respectively, so that $t:=\gamma_p(t_0)\in G(p)$ is an order 4 or 3 stabilizer of $x$.

Set $x=gy$, for $g\in G(p)$.  Note that by the $\CM$ condition, $g$ is in the normalizer of the stabilizer of $x$ (and of $y$).  The commutator subgroup is index two in the normalizer, and $\xi$ is in $\CM^+$ if and only if $g$ is in the trivial coset.  Take a fixed $h_0\in \SL_2(\Z)$ such that $\gamma_p(h_0)$ is in the nontrivial coset and let $h$ be either the identity or $h_0$, so that $\gamma_p(h)^{-1}g$ is in the commutator subgroup. Then letting $g_0=\gamma_p(h)^{-1}g$, 
we see that $g_0$ commutes with $t$.

  Since $d(\xi,\xi')<2\delta\rho_{X(p)}$, it follows from Lemma \ref{c2dist} that there are $M_x,M^0_y\in\SL_2(\Z)$ with $h(M_x)=O(p^{8\delta})$ and 
  $h(M^0_y)=O(p^{8\delta})$ such that $\gamma_p(M_x)\cdot x=x'$ and $g^{-1}\gamma_p(M^0_y)\cdot x=y'$.  
  Then $\gamma_p(M_x(M^0_y)^{-1})g$ maps $ y'$ to $x'$, and so it is in the normalizer subgroup of the stabilizer of $x'$. 
  Thus, $\gamma_p(M_xhM_x^{-1})^{-1}\gamma_p(M_x(M_y^0)^{-1})g$
  commutes with $\gamma_p(M_x)t\gamma_p(M_x)^{-1}$.   
  Setting $M_y=h^{-1}M^0_yh$, this is equivalent to $\gamma_p(M_y)^{-1}g_0\gamma_p(M_x)$ commuting with $t$.
  
Using the identification $G(p)\cong\PGL_2(\F_p)$, we equivalently have the two relations $[t_0,g]=\mathbf{1}$ and $[t_0,M_y^{-1}gM_x]=\mathbf{1}$, which we view as a set of linear equations (defined over $\Z$) in the coefficients of a $2\times 2$ matrix $g\in M_2(\F_p)$ over $\F_p$.

\begin{lemma}

For large enough $p$, the set of solutions $g\in M_2(\F_p)$ to $[t,g]= \mathbf{1}$ and $[t,M_y^{-1}gM_x]=\mathbf{1}$ is at most 1-dimensional.

\end{lemma}

\begin{proof}

Note that these linear equations have coefficients of size $p^{O(\delta)}.$ Let $\tilde{t}$ be a lift of $t$ to $\GL_2(\F_p)$.  The relation $g\tilde t=\tilde t g$ has a 2-dimensional set of solutions in $g$, equal to the span $\langle 1,\tilde t\rangle$ of the identity matrix and $\tilde t$. Thus, either the two relations define a line, or the second relation is redundant, and the second case is equivalent to
 $$\langle M_y^{-1}M_x, M_y^{-1}\tilde{t}M_x\rangle\subset \langle 1,\tilde{t}\rangle.$$
 
It thus follows that $M'=M_y^{-1}\tilde{t}M_y= M_y^{-1}\tilde{t}M_x\cdot(M_y^{-1}M_x)^{-1}$ is also in the span of $1$ and $\tilde{t}$. Since $M'^2 =-1$ it follows that $M'$ is either equal to $\tilde{t}$ or $-\tilde{t}$.  Since $M_y^{-1}t_0M_y$ has coefficients of size $p^{O(\delta)}$, it follows that $M_y^{-1}t_0M_y$ is either $t_0$ or $-t_0$, and similarly for
$M_x^{-1}t_0M_x$.
 
Setting $H$ to be the centralizer of $t_0$, we must have $$M_x^{-1}HM_x =M_y^{-1}HM_y = H.$$ Now, we claim that the elements of the normalizer of $H$ in $\GL_2(\Q)$ which have positive determinant consist exactly of $H$. To prove this, note that its enough to check it after tensoring with $\R$, in which case $H$ becomes an embedded $\C^*\subset \GL_2(\R)$  (unique up to conjugation). Thus, as $M_x,M_y$ have determinant 1, we conclude that $M_x,M_y\in H$.
Finally, note that since $H\cap M_2(\Z)$ is isomorphic to either $\Z[i]$ or $\Z[e^{2\pi i/3}]$  we must have that both $\gamma_p(M_x),\gamma_p(M_y)$ are stabilizers of $x$, contradicting the assumption that our Heegner CM points were distinct.

\end{proof}

Thus, the two relations must not be redundant, and we end up with a single projective solution $g$, which must lie in the span of $\textbf{1},\tilde{t}$. Now, as finding a kernel of a linear map is polynomial in the entries of a map, we can find an integral representative for $g$ with entries of size $p^{O(\delta)}$. Thus $g=a + b\left(\begin{smallmatrix}0&1\\-1&0\end{smallmatrix}\right)$ where $\max(a,b)= p^{O(\delta)}$, and $(x,y)$ lies on $T_m$ where $m=\det g = a^2+b^2$.

\end{proof}

\begin{remark}Note that because isogenies preserve the Weil pairing up to a scalar, Hecke curves do not pass through anti-Heegner $\CM$ points.  Thus, Proposition \ref{repulsion} implies in particular that \emph{all} anti-Heegner CM points repel.
\end{remark}
\newpage
\mysect{Repulsion of singular bicusps}\noindent

\begin{prop}\label{repulsionbicusp}
For all sufficiently small $\delta>0$ and all sufficiently large $p$, if $\xi\in X(p)\times X(p)$ is a point that is within $(1/2+\delta)\rho_{X(p)} $ of at least 3 singular bicusps, then
all the singular bicusps within $(1/2+\delta)\rho_{X(p)}$ of $\xi$ lie on the same Hecke divisor $T_m$ with $m=p^{O(\delta)}$.
\end{prop}

\begin{proof}

Suppose $\xi=(x,y)$ is a point as in the proposition. We first claim that each of $x,y$ is within $(1/2+\delta)\rho_{X(p)} $ of at least 2 distinct cusps. Suppose not, so that without loss of generality there exist a unique cusp $c_0$ within $(1/2+\delta)\rho_{X(p)}$ of $x$. That means that $y$ is within $(1/2+\delta)\rho_{X(p)} $ of at least 2 distinct cusps $c,c'$ which have the same stabilizer as $c_0$. Without loss of generality, by acting with $G(p)$ we can assume that $c = iy_p$. Thus, denoting $H=\langle \sigma_p \rangle$, by Lemma \ref{cuspdist} we have elements $\sigma\in H, M\in\SL_2 \Z$ such that
$h(M)=p^{O(\delta)}$ and $c'=\sigma\gamma_p(M)\cdot c$. Therefore $\sigma\gamma_p(M)\in N(H)$, and so $\gamma_p(M)\in N(H)$. Since $h(M)=p^{O(\delta)}$ and the only upper triangular matrices in $\SL_2 \Z$ are strictly upper triangular, it follows that $\gamma_p(M)\in H$. Thus $c=c'$, which is a contradiction.

Thus each of $x,y$ is within $(1/2+\delta)\rho_{X(p)} $ of at least 2 distinct cusps, and also within $\delta\rho_{X(p)} +O(1)$ of a pre-image of $q_2$. At the cost of decreasing $\delta$ by a factor of $2$ and acting by $G(p)$
we can assume that $x=\iota$ and $y=gx$ for some $g\in G(p)$. Now suppose $(c,c')$ is a singular bicusp that is within $(1/2+\delta)\rho_{X(p)} $ of $(x,y)$. Then it follows similarly to Lemmas \ref{cuspdist} and \ref{c2dist} that there exist
elements $M_1,M_2\in\SL_2(\Z)$ with $h(M_1),h(M_2) = p^{O(\delta)}$ such that $c = \gamma_p(M_1)\cdot iy_p$ and $c' = g\gamma_p(M_2)\cdot iy_p$. It thus follows that $$\gamma_p(M_1)^{-1}g\gamma_p(M_2)\in N(H),$$ or alternatively that
$$g\in \gamma_p(M_1)N(H)\gamma_p(M_2)^{-1}.$$ Now, note that $G(p)$ acts on $\P^1(\F_p)$ and $\gamma_1N(H)\gamma_2^{-1}$ are exactly those elements of $G(p)$ that take $\gamma_2\infty$ to $\gamma_1\infty$ for any $\gamma_1,\gamma_2\in G(p)$. Thus the intersection of any finite number of double cosets of the form $\gamma_1N(H)\gamma_2^{-1}$ is given by specifying the images of finitely many points in $\P^1(\F_p)$. Let $A(g)$ denote the set of pairs $(M_1,M_2)$ of matrices $M_1,M_2\in \SL_2(\Z)$ for which $g\in \gamma_p(M_1)N(H)\gamma_p(M_2)^{-1}$ and consider the intersection $$B(g):=\bigcap_{(M_1,M_2)\in A(g)} \gamma_p(M_1) N(H)\gamma_p(M_2)^{-1}.$$ There are 2 cases:

\begin{enumerate}

\item $B(g)$ specifies where at most 2 distinct points go, but no more. This means that for all the $M_2$ that occur in $A(g),\gamma_p(M_2)$ lies in at most 2 right $N(H)$ orbits.
Now, let $(M_1,M_2),(M_1',M_2')$ be two pairs in $A(g)$ for which $\gamma_p(M_2),\gamma_p(M_2')$ are in the same right $N(H)$-orbit. Then $\gamma_p(M_2'^{-1}M_2)\in N(H)$, and since $h(M_2'^{-1}M_2)=p^{O(\delta)}$, it follows that for large enough $p$, $M_2$ and $M_2'$ must be in the same right $H$-orbit, and similarly for $M_1,M_1'$. But this means that they correspond to the same bicusp, and thus there are at most 2 distinct singular bicusps within $(1/2+\delta)\rho_{X(p)} $ of $\xi$, which is a contradiction.

\item $B(g)$ specifies where 3 distinct points go, and thus specifies $g$. Note that this means that there is a representative for $g$ which is polynomial in the coefficients of the linear equations mod $p$, all of whose coefficients are reductions of elements in $\Z$ of size $p^{O(\delta)}$. Thus there is an integral representative for $g$ with entries of size $p^{O(\delta)}$. Since $c' = g\gamma_p(M_2M_1^{-1})\cdot c$ it follows that $(c,c')$ is on some $T_m$ with $m=\det g = p^{O(\delta)}$, as desired.

\end{enumerate}

\end{proof}

\mysect{Repulsion of diagonals}\label{diagsect}
Let $\pi_{i}:(X(p)\times X(p))^{2}\into X(p)\times X(p)$ be the projection onto the $i$th factor; we denote a point $\xi\in (X(p)\times X(p))^2$ by $\xi=(x_1,y_1,x_2,y_2)$ where $\pi_i(\xi)=(x_i,y_i)$ for $i=1,2$.  By the big diagonals of $(X(p)\times X(p))^2$ we mean the subvarieties of the form
\[\Delta_g=\{(gx,gy,x,y)\}\subset (X(p)\times X(p))^2\]
for some fixed $g\in G(p)$, and by a small Hecke curve of degree $m$ we'll mean
\[\tau_{g,m}=\{(gx,gy,x,y)\mid (x,y)\in T_m\}\subset \Delta_g\]
In the following proposition we'll be concerned with the Kobayashi neighborhoods of the big diagonals, which take the form
\[B(\Delta_g,R)=\{(x_1,y_1,x_2,y_2)\mid d(x_1,gx_2)<2R\textrm{ and }d(y_1,gy_2)<2R\}\]
\begin{prop}\label{repulsiondiag}  For all sufficiently small $\delta>0$ and all sufficiently large $p$, if a point $\xi\in (X(p)\times X(p))^2$ is in $\omega(\log p)$ many distinct neighborhoods
$B(\Delta_g,\delta\rho_{X(p)})$ then one of the following must be true:

\begin{enumerate}
\item[(a)]  $\xi$ is within a distance $(1/2+O(\delta))\rho_{X( p)}$ of a point both of whose projections are singular bicusps. In this case, $\xi$ is in $O(p\cdot \delta \rho_{X(p)}/a(2d)^{1/2})$ many such neighborhoods where $d$ is the
        smaller of the distances of $\pi_1(\xi)$ and $\pi_2(\xi)$ to a singular bicusp;
\item[(b)]  $\xi$ is within a distance $O(\delta)\rho_{X(p)}$ of a small Hecke curve $\tau_{g,m}$ of degree $m=p^{O(\delta)}$.
\end{enumerate}
\end{prop}
Note that with our usual normalization, $O(p\cdot \delta \rho_{X(p)}/a(2d)^{1/2})=O(p^{1+\delta}e^{-d})$ by Corollary \ref{biginj}.
\begin{proof} Let  $\xi=(x_1,y_1,x_2,y_2)$ be a point in $\omega(\log p)$ many neighborhoods $B(\Delta_g,\delta\rho_{X(p)})$.  We split the proof up into 2 cases:
\subsubsection*{\underline{Case 1:}}

First suppose one of the coordinates, say $x_1$, is within a distance $(1/2-3\delta)\rho_{X(p)}-O(1)$ of a cusp $c$, which after acting by an element of $G(p)$ we may assume to be the image of $iy_p$.  For each $g\in G(p)$ such that $d(x_1,gx_2)<2 \delta\rho_{X(p)}$ it must be the case that the cusp nearest to $gx_2$
is also $c$, by Lemma \ref{disksep}. Hence, since $\xi$ is in many diagonal neighborhoods, $g$ must be in a right $H$ coset, where $H$ is the stabilizer of $c$. Without loss of generality, we can similarly assume the nearest cusp to $x_2$ is $c$ and therefore the set of all $g$ such that $\xi\in B(\Delta_g,\delta \rho_{X(p)})$ is inside $H$.

Next, let $c'=g_0c$ be the cusp closest to $y_1$, and assume that $y_1$ is not within $(1/2+2\delta)\rho_{X(p)}$ of any cusp stabilized by $H$.  Then $y_1$ can't be within
$(1/2-2\delta)\rho_{X(p)}$ of $c'$, or else there would be no $h\in H$ such that $d(y_1,hy_1)< 4\delta\rho_{X(p)}$, which must be the case since $\xi\in B(\Delta_h,\delta\rho_{X(p)})\cap B(\Delta_1,\delta\rho_{X(p)})$. Thus $y_1$ is within $2\delta\rho_{X(p)} +O(1)$ of a point projecting to $q_2$, which we may assume to be $g_0\iota$. Now, for each $h\in H$ such that $\xi\in B(\Delta_h,\delta\rho_{X(p)} )$, we have $d(g_0\iota,hg_0\iota)<8\delta\rho_{X(p)}+O(1)$, so by Lemma \ref{c2dist} there is a matrix $M\in\SL_2(\Z)$ with $h(M)=O(p^{32\delta})$ such that $\gamma_p(M)=g_0^{-1}hg_0$.  $M$ is unipotent, and has a fixed vector $v$ whose coordinates have size $p^{O(\delta)}$, so we can find $M'\in \SL_2(\Z)$ with $h(M')=p^{O(\delta)}$ sending $v$ to $\infty$, and thus we have $g_0\in N(H)\gamma_p(M')$.  As $g_0\iota$ is within $(1/2+ O(\delta))\rho_{X(p)} $ of a cusp stabilized by $H$, $y_1$ is as well, and $\pi_1(\xi)$ is within a distance $R\leq (1/2+O(\delta))\rho_{X(p)}$ of a singular bicusp.  Since $\xi$ is close to some diagonal, $x_2$ is within $(1/2-\delta)\rho_{X(p)} -O(1)$ of a cusp, so running the same argument (after shrinking $\delta$), we also get that $\pi_2(\xi)$ is within $R$ of a singular bicusp.

 Finally, under any projection to $X(p)$, let $x$ be the image of $\xi$ and $c$ the image of the nearby singular bicusp (a distance $d$ away).  If $\sigma\in G(p)$ is a generator of the stabilizer of $c$, then the images $\sigma^kx$ equidistribute around the boundary of the ball $B(c,d)$ of radius $d$ (note that $d$ is less than the injectivity radius).  $B(c,d)$ has circumference of length $O(a(2d)^{1/2})$, so there are $O(p\cdot \delta \rho_{X(p)}/a(2d)^{1/2})$ images within a distance $2\delta\rho_{X(p)}$ of $x$.
\subsubsection*{\underline{Case 2:}}

Now assume that none of the coordinates of $\xi$ is within a distance $(1/2-3\delta)\rho_{X(p)}-O(1)$ of a cusp, and we show that $\xi$ must be within $O(\delta\rho_{X(p)} )$ of a small Hecke curve.  After shrinking $\delta$, we can assume $\xi=(\gamma_1 \iota,\gamma_2 \iota,\gamma_3 \iota,\gamma_4 \iota)$, since $\xi$ is within a radius of $O(4\delta\rho_{X(p)})$ of such a point.  For each $g$ such that $\xi\in B(\Delta_g,\delta\rho_{X(p)})$, by Lemma \ref{c2dist} there exist $M_g,M_g'\in\SL_2(\Z)$ with $h(M_g),h(M_g')=O(p^{8\delta})$ such that $\gamma_p(M_g)=\gamma_1^{-1}g\gamma_3$ and $\gamma_p(M_g')=\gamma_2^{-1}g\gamma_4$, or in other words that
\[\gamma_1\gamma_p(M_g)\gamma_3^{-1}=\gamma_2\gamma_p(M_g')\gamma_4^{-1}\]
Two distinct such elements $g,h$ would then yield matrices $N={M_h}^{-1}M_g$ and $N'={M'_h}^{-1}M'_g$ with $h(N),h(N')=O(p^{16\delta})$ such that
\[\gamma_3\gamma_p(N)\gamma_3^{-1}=\gamma_4\gamma_p(N')\gamma_4^{-1}\]
or equivalently
\[\gamma_p(N)=\gamma \gamma_p(N')\gamma^{-1}\]
for $\gamma=\gamma_3^{-1}\gamma_4$.

Defining $$S_\delta:=\{M\in \SL_2(\Z)\mid h(M)=O(p^{16\delta})\},$$
and $\bar S_\delta\subset \PSL_2(\F_p)$ its reduction $\textmod$ $p$, we see that the number of diagonal neighborhoods containing $\xi$ is bounded by
\[|\gamma_p(\bar S_\delta)\cap \gamma \gamma_p(\bar S_\delta)\gamma^{-1}|.\]

Now consider the larger set $S'_{\delta} = \{M\in M_2(\Z)\mid h(M)=O(p^{6\delta})\}$, $\bar S'_\delta\subset M_2(\F_p)$ its reduction, and the subspace
$$T:=\Span(\gamma_p(\bar S'_{\delta})\cap \gamma \gamma_p(\bar S'_{\delta})\gamma^{-1})\subset M_2(\F_p).$$ We now separate into two cases:

\begin{enumerate}

\item  The centralizer of $T$ consists of more than just scalars.  It follows that $T$ is a sub-algebra, and so it must either be a torus, or isomorphic to $\F_p[x]/(x^2)$.  If $T\cong \F_p[x]/(x^2)$ then all the elements in $$\gamma_p(\bar S_\delta)\cap \gamma \gamma_p(\bar S_\delta)\gamma^{-1}$$ are in a single unipotent subgroup $U$. Thus as in the analysis of $y_1$ in Case 1 all the coordinates of $\xi$ must be within $(1+O(\delta))\log p$ of a cusp stabilized by $U$ and we are in part (a) of the proposition.

If $T$ is a torus, by picking a non-scalar element $\gamma_p(M)\in T$ we can lift $T$ to a torus $\tilde{T}\subset \textrm{M}_2(\Q)$ spanned by $\textbf{1}$ and $M$. Thus the elements in $$\gamma_p(\bar S_\delta)\cap \gamma \gamma_p(\bar S_\delta)\gamma^{-1}$$ are reductions of elements in the norm 1 subgroup of $\tilde{T}$, and hence are generated by a single semisimple element $\gamma_p(M)$. As there are at most $O(\log p)$ elements $M^k$ with height bounded by $p^{O(\delta)}$, it cannot be the case that $\xi$ is in $\omega(\log p)$ many diagonal neighborhoods.

\item The centralizer of $T$ consists of scalars. This means we can pick at most three elements in $A_1,A_2,A_3\in T$ such that they have no common centralizer outside of scalars.
Thus, $\gamma$ is determined projectively by the three elements $\gamma A_i\gamma^{-1}$. Since $h(A_i)=O(p^{6\delta})$, this means we can find a projective representative $\tilde{\gamma}\in\GL_2(\F_p)$ for $\gamma$ with entries of size $p^{O(\delta)}$. Thus, by Gaussian elimination and the Euclidean algorithm, we can find elements $M_1,M_2\in \SL_2(\Z)$ of height $p^{O(\delta)}$ such
that $$\gamma_p(M_1) \tilde{\gamma}\gamma_p(M_2)=\mat{0}{m}{-1}{0},$$ where $m=\det\tilde{\gamma} = p^{O(\delta)}.$

Next, note that $d_{Y(p)}(i,i\sqrt{m}) = \frac12\log m$, and that $$\left(i\sqrt{m},\mat{0}{m}{-1}{0}\cdot i\sqrt{m}\right)\in T_m.$$ Thus, since the metric $h_{Y(p)}$ on $Y(p)$ is strictly smaller than the metric $h_{X(p)}$, by the above and Proposition \ref{metcom1} we have

\begin{align*}
d_{X(p)}((\gamma_3\iota,\gamma_4\iota),T_m) &= d_{X(p)}((\iota,\gamma\iota),T_m)\\
&\leq d_{X(p)}(\gamma_p(M_1^{-1})\iota,\iota) + d_{X(p)}(\gamma_p(M_2)\iota,\iota) + d_{X(p)}((\gamma_p(M_1)\iota,\gamma\gamma_p(M_2)\iota),T_m)\\
&\leq O(\delta\rho_{X(p)} )+ d_{Y(p)}((\iota,\left(\begin{smallmatrix}0 & m \\ -1 &0\end{smallmatrix}\right)\iota),T_m)\\
&\leq O(\delta\rho_{X(p)}) + 2d_{Y(p)}(i,i\sqrt{m}) \\
&\leq O(\delta\rho_{X(p)})\\
\end{align*}
\noindent which establishes the claim, since $m=p^{O(\delta)}$.

\end{enumerate}

\end{proof}

\section{Volume estimates}\label{volest}
In this section we prove that, for certain special subvarieties $Z$ of hyperbolic manifolds, the total volume of a curve
in a tubular neighborhood of $Z$ of radius $r$ grows sharply as a function of $r$. This has two consequences:  one can effectively bound the volume in a radius $r$ tube by the volume in a larger radius $R>r$ tube, and in the limit $r\rightarrow 0$ one can effectively bound the multiplicity of a curve along $Z$ by the volume in a radius $R$ tube.  In both cases, bigger neighborhoods give better bounds.


\mysect{Global volume estimates}

Let $X$ be a hyperbolic curve and consider any curve $C\subset X\times X$.  Work of Hwang and To \cite{hwangto1,hwangto2} provides a bound on the multiplicity of $C$ at a point $\xi\in X\times X$ in terms of the volume of $C$ in a Kobayashi ball centered at $x$.  Similarly, the multiplicity of $C$ along the diagonal $\Delta\subset X\times X$ is bounded by its volume within the Kobayashi tubular neighborhood
\[B(\Delta, r):=\{(x,y)\mid d(x,y)<2r\}\subset X\times X\]
of the diagonal.  For $r<\rho_X$, $B(\Delta,r)$ is the quotient of the neighborhood
\[B(\Delta_\D,r)=\{(z,w)\mid d(z,w)<2r\}\subset \D\times\D\]
by the diagonal action of $\pi_1(X)$.

We employ the conventions from the Introduction.  In particular, recall that we define $a(r):=\vol B_\D(0,r)$ to avoid dependence on the normalization in the following theorems, though for computations we always take the curvature of $\D$ to be $-1$.  We then have:

\begin{theorem}\label{HT}  For any curve $C\subset X\times X$:
\begin{enumerate}
\item[(a)]\hspace{.01in}\cite[Theorem 2 ]{hwangto1}  For any point $\xi\in X\times X$, and $r<\rho_X$, then
\[\frac{1}{a(r)}\vol(C\cap B(\xi,r))\geq \mult_\xi (C).\]

\item[(b)]\hspace{.01in}\cite[Theorem 1]{hwangto2}  For any $r<\rho_X/2$,
\[\frac{1}{a(r)}\vol(C\cap B(\Delta,r))\geq 2(C\cdot\Delta).\]

\end{enumerate}
\end{theorem}
Both statements in Theorem \ref{HT} are optimal in the sense that the bound is realized:  by a union of fibers in part (a) and by a union of diagonal translates of the graph of $-z:\D\into\D$ in part (b).  For the convenience of the reader, we summarize a different proof of part (b) of Theorem \ref{HT} than that given in \cite{hwangto2} as the same framework will also yield the relative bounds we require.  

Any point $(z,w)\in \D\times\D$ lies on a diagonal translate of the graph of $-z:\D\into \D$ and it will be convenient to define a function
\[\mu(z,w):=\tanh^2(d_{\mathbb{D}}(z,w)/4)\]
measuring the Euclidean distance of $(z,w)$ from 0 in this ``antidiagonal" disk.  Indeed, we have
\[\chi(z,w):=\tanh^2 (d_{\mathbb{D}}(z,w)/2)=\left| \frac{w-z}{1-\bar z w}\right|^2.\] 
\begin{lemma}\label{mupluri}$\log\mu$ is plurisubharmonic.
\end{lemma}
\begin{proof}This follows from a direct computation.  Since we have
\begin{equation}\log \mu=-2\tanh^{-1}\sqrt{1-\chi}\label{whatmu}\end{equation}
the plurisubharmonicity of $\log\mu$ also follows from the criterion of \cite[Lemma 5]{hwangto2} (indeed, the function on the right hand side of \eqref{whatmu} is chosen to satisfy the differential equation therein).
\end{proof}
It follows that the function $F=-8\pi\log(1-\mu)$ considered by Hwang and To is plurisubharmonic and satisfies $\omega_{\D\times\D}\geq dd^c F$.  Both functions $\mu,\chi$ are diagonally invariant under the action of $\SL_2\R$ and therefore descend to $B(\Delta,r)\subset X\times X$ provided $r<\rho_X/2$.  Define
\[I(r):=\int_{C\cap B(\Delta,r)}dd^c F.\]
Note that on the one hand $\vol(C\cap B(\Delta,r))\geq I(r)$, while on the other hand we can show that $I(r)$ grows at least as fast as $a(r)$:
\begin{lemma}\label{increasinglemma}  For $r<\rho_X/2$, $\frac{1}{a(r)}I(r)$ is an increasing function of $r$.
\end{lemma}
\begin{proof}Set $f(s)=-8\pi\log(1-e^s)$.  We have by Stokes' theorem
\begin{align}
I(r)&= \int _{C\cap B(\Delta,r)}dd^cF\notag\\
&=f'(\log \tanh^2(r/2))\int_{C\cap B(\Delta,r)} dd^c \log \mu\notag\\
&=  8\pi\sinh^2(r/2)\int_{C\cap B(\Delta,r)}dd^c\log \mu.\label{lastline}
\end{align}
Indeed, since $f$ is indistinguishable at the boundary of $B(\Delta,r)$ to a linear function of slope $f'(r)$, we can approximate $f$ by such a function without changing the integral on the interior.  As $\log\mu$ is plurisubharmonic, from \eqref{lastline} it follows that $\frac{1}{\sinh^2(r/2)}I(r)$ is an increasing function and the claim follows.
\end{proof}
We therefore have
\[\vol(C\cap B(\Delta,r))\geq I(r)\geq a(r)\cdot\lim_{r\into 0}\frac{1}{a(r)}I(r)\]
and to conclude Theorem \ref{HT}(b), we need only compute the limit.  Using \eqref{lastline} we have
\[\lim_{r\into 0}\frac{1}{a(r)}I(r)=2 \cdot \lim_{r\into 0}\int_{C\cap B(\Delta,r)}dd^c \log\mu\]
and by a local computation (see \cite{hwangto2} for details) the right hand side is bounded by twice the multiplicity along the diagonal,
\begin{equation}\lim_{r\into 0}\int_{C\cap B(\Delta,r)}dd^c \log\mu\geq (C\cdot\Delta).\end{equation}

We will also require an analogue of the above theorem for the diagonal
\[\Delta_2=\{((x,y,x,y)\}\subset (X\times X)^2.\]
Around $\Delta_2$ we have the (Kobayashi) tubular neighborhood considered in the previous section
\[B(\Delta_2,r)=\{(x_1,x_2,y_1,y_2)\mid d(x_1,y_1)<2r\textrm{ and } d(x_2,y_2)<2r\}\subset (X\times X)^2\]
for any $r<\rho_X/2$, and it is the quotient of the analogous diagonal neighborhood $ B(\Delta_2,r)\subset (\D\times\D)^2$ by the diagonal action of $\pi_1(X)^2$.  We thank the referee for providing a more streamlined version of the authors' original argument for the following
\begin{lemma}\label{htbd} For $X,r$ as in Theorem \ref{HT} and for any curve $C\subset (X\times X)^2$ not contained in $\Delta_2$, we have
\[\frac{1}{a(r)}\vol(C\cap B(\Delta_2,r))\geq 2\sum_{\xi\in\Delta_2}\mult_\xi (C).\]

\end{lemma}
\begin{proof} Let $\pi_i: (X\times X)^2\into X\times X$ be the projections onto the $x_i,y_i$ coordinates respectively, and let $F_i=\pi_i^* F$ and $\mu_i=\pi^*_i\mu$.  Also set $M=\tanh^2(r/2)$ and consider the integral
\[I_i(r):=\int_{C\cap \{\mu_j< \mu_i< M\}}dd^c F_i.\]
Stokes' theorem gives
\begin{align*}
I_i(r)&=\int_{C\cap\{\mu_j< \mu_i=M\}}d^c F_i-\int_{C\cap\{\mu_j= \mu_i< M\}}d^c F_i\\
&=8\pi\sinh^2(r/2)\int_{C\cap\{\mu_j< \mu_i=M\}}d^c \log \mu_i -\int_{C\cap\{\mu_j= \mu_i< M\}}d^c F_i
\end{align*}
but since
\[\int_{C\cap\{\mu_j< \mu_i< M\}}dd^c \log \mu_i =\int_{C\cap\{\mu_j< \mu_i=M\}}d^c \log \mu_i -\int_{C\cap\{\mu_j= \mu_i< M\}}d^c \log \mu_i \]
bounds the multiplicity of $C$ at all points $\xi\in\Delta_2$ in the closure of $\{d(x_j,y_j)< d(x_i,y_i)< 2r\}$, we have 
\begin{align*}\frac{1}{2a(r)}(I_1(r)+I_2(r))&\geq \sum_\xi \mult_\xi (C)+\sum_i\int_{C\cap\{\mu_j= \mu_i< M\}}d^c\log\mu_i-\frac{1}{8\pi\sinh^2(r/2)}d^cF_i\\
&=\sum_\xi \mult_\xi (C)+\sum_i\int_{C\cap\{\mu_j= \mu_i< M\}}\left(1-\frac{\sinh^2(d(x_i,y_i)/4)}{\sinh^2(r/2)}\right)d^c\log\mu_i\\
&=\sum_\xi \mult_\xi (C)+\int_{C\cap\{\mu_2= \mu_1< M\}}\left(1-\frac{\sinh^2(d(x_1,y_1)/4)}{\sinh^2(r/2)}\right)d^c\log\mu_1/\mu_2
\end{align*}
the last line taking into account the induced orientations.  The integrand on the right is positive (as $\log\mu_1/\mu_2$ cuts out $C\cap \{\mu_2=\mu_1\}$ on $C$), so the claim follows by taking the $r\into 0$ limit of the right hand side, since we certainly have 
\[\vol(C\cap B(\Delta_2,r))\geq I_1(r)+I_2(r).\]
\end{proof}


\mysect{Relative volume estimates}
We now refine the strategy of Lemma \ref{increasinglemma} to gain better control over the growth of the volume of a curve $C\subset X\times X$ contained within a tube around the diagonal $\Delta\subset X\times X$.

\begin{prop}\label{htd}Let $X$ be a compact hyperbolic complex curve and $C\subset X\times X$ a complex curve that is not the diagonal.  Then for $r<\rho_X/2$,
\[\frac{1}{\cosh(r)}\vol(C\cap B(\Delta,r))\]
is an increasing function of $r$.
\end{prop}
\begin{remark}
The coefficient in Proposition \ref{htd} is presumably not optimal, but we will only care about its asymptotic behavior.  Note that the statement requires the curvature to be $-1$; for the metric of constant sectional curvature $-\frac{1}{\lambda^2}$, we would have that

\[\frac{1}{\cosh(r/\lambda)}\vol(C\cap B(\Delta,r))\]
is increasing.
\end{remark}
\begin{proof}
Let $f(s)=\log\left(\frac{s}{1-s}\right)$.  We have as currents
\begin{align}
dd^c(f\circ\chi)&=-\frac{1}{2\pi}dd^c \log (|1-\bar z w|^2-|z-w|^2)+dd^c\log|z-w|^2\notag\\
&=-\frac{1}{2\pi}dd^c\log \left[(1-|z|^2)(1-|w|^2)\right]+[\Delta_{\mathbb{D}}] \notag\\
&=\frac{1}{4\pi}\omega_{\D\times\D}+[\Delta_{\mathbb{D}}].\notag
\end{align}
Since $\chi=\frac{4\mu}{(1+\mu)^2}$, we have $\frac{\chi}{1-\chi}=\frac{4\mu}{(1-\mu)^2}$ and
\[g(\log\mu):=\log\left(\frac{4\mu}{(1-\mu)^2}\right)=f\circ\chi.\]
As in Lemma \ref{increasinglemma}, by Stokes' theorem we have
\begin{align}
J(r):=\frac{1}{4\pi}\vol(C\cap B(\Delta,r))+(C.\Delta)&= \int _{C\cap B(\Delta,r)}dd^c(f\circ \chi)\notag\\
&=g'(\log \tanh^2(r/2))\int_{C\cap B(\Delta,r)}dd^c\log\mu\notag\\
&=  \cosh(r)\int_{C\cap B(\Delta,r)}dd^c\log\mu.\notag
\end{align}
By Lemma \ref{mupluri}, $\log\mu$ is plurisubharmonic, so $\frac{1}{\cosh(r)}J(r)$ is an increasing function and the claim follows.
\end{proof}

\section{Multiplicity    estimates}\label{mult}
Consider the product $(X(p)\times X(p))^n$.  Denote by $\pi_i$ the projection onto the $i$th copy of $X(p)\times X(p)$ and $\pi_{ij}:=\pi_i\times\pi_j$.  For the proof of Theorem \ref{gonality}, we will need to control the ramification of curves $C\subset (X(p)\times X(p))^n$ over their image in $\Sym^n Z(p)$.  Such ramification occurs when $C$ passes through one of the sets
\[\begin{array}{cccc}\pi_i^{-1}(\CM),&\pi_i^{-1}(\CUSP),&\pi_{ij}^{-1}(\Delta_g),\end{array}\]
for some $i,j$, where $\Delta_g\subset (X(p)\times X(p))^2$ is the diagonal considered in Section \ref{diagsect}.
Recall that for $C\subset (X(p)\times X(p))^n$ we define 
$$\Deg (C):=K_{(X(p)\times X(p))^n}\cdot C$$
Note that $\Deg (C)=\frac{1}{2\pi}\vol(C)$ with our usual normalization.  In this section we prove that incidence of $C$ along each of these sets is negligible with respect to $\Deg(C)$ for large $p$.  For any set $S$ of (closed) points, let $\mult_S(C)=\sum_{x\in S}\mult_x(C)$.  Recall that by $T_m^*C$ we mean the pullback of the divisor $C$ along the Hecke correspondence in the second variable; clearly $\Deg(T_m^* C)=\deg (T_m)\cdot \Deg(C)$.  

By Corollary \ref{biginj}, $\rho_{X(p)}\sim 2\log p$.  Throughout this section the key observation is that for any fixed $t>0$, for sufficiently large $p$ we have 
\[\frac{\vol(C)}{a(t\cdot \rho_{X(p)})}=O(p^{-2t}\Deg (C)).\] 

\mysect{Multiplicity in $X(p)\times X(p)$}

\begin{prop} \label{ramification} For all sufficiently small $\delta>0$ and all sufficiently large $p$, and for any non-Hecke curve $C\subset X(p)\times X(p)$, $$\mult_{\CM} (C)=O(p^{-\delta}\Deg(C)).$$\end{prop}
\begin{proof}  For $d=p^\delta$, partition $\CM$ into two sets \[T:=\CM\cap\cup_{m<d}T_m\hspace{.2in}
\mathrm{and}\hspace{.2in}S:=\CM- T.\]  By Proposition \ref{repulsion}, for sufficiently small $\delta'>0$ the balls $B(\xi,\delta'\rho_{X(p)})$ are disjoint as $\xi$ varies over $S$. Using Corollary \ref{biginj}, we then have that
\begin{align*}
\mult_{\CM}(C)&= \mult_S(C) + \mult_T(C)\\
&\ll \frac{1}{a(\delta'\rho_{X(p)})}\sum_{\xi\in S}\vol\left(C\cap  B(\xi,\delta'\rho_{X(p)})\right) + \sum_{m<d} (C.T_m)\\
&\ll \frac{\vol(C)}{a(\delta'\rho_{X(p)})} + \sum_{m<d} (T_m^*C.\Delta)\tag{Thm. \ref{HT}(a)}\\
&\ll \frac{\vol(C)}{a(\delta'\rho_{X(p)})} + \frac{\vol(C)}{a(\rho_{X(p)}/2)}\sum_{m<d}\deg (T_m)\tag{Thm. \ref{HT}(b)}\\
&\ll (p^{-2\delta'}+d^3p^{-1})\Deg(C)
\end{align*}
since $\deg (T_m)=O(d^2)$ and the result follows.
\end{proof}

\begin{prop} \label{ramificationcusp}
For all sufficiently small $\delta>0$, all sufficiently large $p$, and for any non-Hecke curve $C\subset X(p)\times X(p)$, $$p\mult_{\CUSP}(C)= O(p^{-\delta}\Deg(C)).$$
\end{prop}
\begin{proof}

Take $d=p^\delta$, and again partition the points of $\CUSP$ into
 \[T=\CUSP\cap\cup_{m<d}  T_m\hspace{.2in}
\mathrm{and}\hspace{.2in} S=\CUSP - T.\]
By Proposition \ref{repulsionbicusp}, for sufficiently small $\delta'>0$ any point in $X(p)\times X(p)$ is in at most two of the balls $B(\xi,(1/2+\delta')\rho_{X( p)})$ for $\xi\in S$.  By Theorem \ref{HT}(a), for each $\xi\in \CUSP$,
\[\mult_{\xi}(C) \ll  \frac{1}{a((1/2+\delta')\rho_{X( p)})}\vol(C\cap B(\xi,(1/2+\delta')\rho_{X( p)})))\]
 and it therefore follows that
 \begin{align}
\mult_{S}(C)
&\ll \frac{\vol(C)}{a((1/2+\delta')\rho_{X( p)})}  \notag\\
&\ll p^{-1-2\delta'}\Deg(C) .\label{last2}
\end{align}
Now for any $m<d$,
\begin{align*}
\mult_{\CUSP\cap T_m}(C) &=\sum_{\xi\in \CUSP\cap T_m} \mult_\xi (C)\\
&\ll  \sum_{\xi\in\Delta\cap \CUSP}\mult_\xi (T_m^* C)\\
&\ll (T_m^*C. \Delta)\\
&\ll \frac{d^2\cdot \vol(C)}{a(\rho_{X(p)}/2)}\\
\end{align*}
by Theorem \ref{HT}(b). Thus,
$$\mult_{T}(C)\ll p^{-1+3\delta}\Deg(C).$$ 
and combining with equation \eqref{last2}, the result follows.

\end{proof}
\mysect{Multiplicity in $(X(p)\times X(p))^2$}

\begin{prop} \label{ramificationdiag}

For all sufficiently small $\delta>0$, all sufficiently large $p$, and any curve $C\subset (X(p)\times X(p))^2$ not contained in any diagonal $\Delta_g$,
$$\sum_g\mult_{\Delta_g} C = O(p^{-\delta}\Deg(C)).$$

\end{prop}

For the proof, we shall need both metrics $h_{Y(p)}$ and $h_{X(p)}$.  We denote volume with respect to $h_{Y(p)}$ by $\vol'$ instead of $\vol$, and likewise when we consider balls in the $h_{Y(p)}$ metric we write $B'$ instead of $B$.  $\vol'$ and $\vol$ are comparable in the following sense:
\begin{lemma} \label{newlemma} For any curve $C\subset X(p)\times X(p)$ and any open set $U\subset X(p)\times X(p)$,
\[\vol'(C\cap U)-\vol(C\cap U)=O(p^{-1}\vol(C))\]
\end{lemma}
\begin{proof}Since the metrics on $X(p)\times X(p)$ are the sum of pullbacks of the metrics on each factor,
$$\vol'(C)=\frac{\vol'(X(p))}{\vol(X(p))}\vol(C).$$  By Gauss--Bonnet, $\vol(X(p))$ (resp. $\vol'(X(p))$) is a multiple of the euler characteristic $\chi(X(p))$ (resp. $\chi(Y(p))$).  We have $\chi(Y(p))=\chi(X(p))-\#(\mathrm{cusps})$, so we compute
\[\vol'(C)-\vol(C)=O(p^{-1}\vol(C))\]
since $\chi(X(p))\sim -p^3$ and $\#(\mathrm{cusps})\sim p^2$.  As $h_{X(p)}\leq h_{Y(p)}$, $\vol'-\vol$ is a positive measure, and the result follows.
\end{proof}  
\begin{proof}[Proof of Proposition \ref{ramificationdiag}]
Let $\pi_i:(X(p)\times X(p))^2\into X(p)\times X(p)$ be the two projections and assume $C$ has larger degree $d$ along the first.  By Proposition \ref{repulsiondiag}, the neighborhoods $B(\Delta_g,\delta\rho_{X(p)})$ only overlap more than $O(\log p)$ times within $O(\delta)\rho_{X(p)}$ of a small Hecke curve $\tau_{g,k}$ with $k=p^{O(\delta)}$ or within $(1/2+O(\delta))\rho_{X(p)} $ of a point which projects to a singular bicusp along both projections $\pi_i$.  Let $E$ be the sum of the volumes of the intersection of $C$ with these latter balls, and let $T_k^c$ denote the points of $T_k$ not within $(1/2-O(\delta))\rho_{X(p)} $ of a singular bicusp. Likewise, denote by $\tau_{g,k}^c$ the points of $\tau_{g,k}$ which are not within $(1/2-O(\delta))\rho_{X(p)} $ of a singular bicusp in either projection $\pi_i$.

We have by Lemma \ref{htbd}
\begin{align*}
\sum_g\mult_{\Delta_g} C&\ll \frac{1}{a(\delta\rho_{X(p)}/2)}\sum_g\vol(C\cap B(\Delta_g,\delta\rho_{X(p)}))\\
&\ll \frac{1}{a(\delta\rho_{X(p)}/2)}\left(O(\log p)\cdot \vol(C) + \sum_{k=p^{O(\delta)}}\sum_g\vol(C\cap B(\tau^c_{g,k},O(\delta)\rho_{X(p)}))+E\right) \\
&\ll p^{-\delta}\Deg (C) + \frac{1}{a(\delta\rho_{X(p)}/2)}\left(d\cdot p^{O(\delta)}\sum_{k=p^{O(\delta)}}\sum_{i}\vol(\pi_{i}(C)\cap B(T^c_k,O(\delta)\rho_{X(p)}))+E\right)\\
\end{align*}
where we have bounded the maximum multiplicity of the overlaps of the neighborhoods $B(\tau^c_{g,k},O(\delta)\rho_{X(p)})$ for fixed $k$ by $p^{O(\delta)}$, using Lemma \ref{c2dist}.
The key observation now is that $T_k$ is an \'etale correspondence on $Y(p)$ and therefore preserves the metric $h_{Y(p)}$.  By Proposition \ref{metcom1}, the two metrics are comparable away from the cusp; precisely, on the set $U\subset X(p)\times X(p)$ of points neither of whose coordinates is within $(1/2-O(\delta))\rho_{X(p)} $ of a cusp, there is a constant $A$ only depending on $\delta$ such that
\[\frac{1}{A}\cdot h_{Y(p)\times Y(p)}|_U\leq h_{X(p)\times X(p)}|_U\leq A \cdot h_{Y(p)\times Y(p)}|_U.\]  Thus, at the cost of increasing the implicit constant in $O(\delta)$ we have $B(T^c_k,O(\delta)\rho_{X(p)}))\subset B'(T^c_k,O(\delta)\rho_{X(p)}))$.  Likewise the volume forms $\vol'$ and $\vol$ are within a constant (only depending on $\delta$) of each other on $U$, so 
\begin{align*}
&\ll p^{-\delta}\Deg(C)+\frac{d\cdot p^{O(\delta)}}{a(\delta\rho_{X(p)}/2)}\sum_{k=p^{O(\delta)}}\sum_i\vol'(\pi_{i}(C)\cap B'(T^c_k,O(\delta)\rho_{X(p)}))+\frac{E}{a(\delta\rho_{X(p)}/2)}\\
&\ll p^{-\delta}\Deg(C)+\frac{d\cdot p^{O(\delta)}}{a(\delta\rho_{X(p)}/2)} \sum_{k=p^{O(\delta)}}\sum_i\vol'(T_k^*\pi_{i}(C)\cap B'(\Delta,O(\delta)\rho_{X(p)}))+\frac{E}{a(\delta\rho_{X(p)}/2)}\\
&\ll p^{-\delta}\Deg(C)+\frac{d\cdot p^{O(\delta)}}{a(\delta\rho_{X(p)}/2)} \sum_{k=p^{O(\delta)}}\sum_i\vol'(T_k^*\pi_{i}(C)\cap B(\Delta,O(\delta)\rho_{X(p)}))+\frac{E}{a(\delta\rho_{X(p)}/2)}\\
\end{align*}
where in going from line 2 to line 3 we've used the upper bound in Proposition \ref{metcom1}.  By Lemma \ref{newlemma}, the middle term above is bounded by
 \begin{align*}\frac{d}{a(\delta\rho_{X(p)}/2)}   \sum_{k=p^{O(\delta)}}\sum_i\left( \vol(T_k^*\pi_{i}(C)\cap B(\Delta,O(\delta\rho_{X(p)}))) + p^{-1}\vol(T_k^*\pi_1(C))\right)\\\ll p^{O(\delta)-1}\Deg(C).\end{align*}
It remains to bound $E$. Note that by Proposition \ref{repulsiondiag}$$E\ll\sum_{\xi\in\CUSP}\sum_{\ell=1}^{(1/2+O(\delta))\rho_{X(p)}} p^{1+\delta}e^{-\ell}\vol(C\cap B(\xi,\ell)).$$

At the cost of increasing $\delta$ by a constant factor, by Proposition \ref{htd} $$E\ll p^{-\delta}\log p\sum_{\xi\in \CUSP} \vol( C\cap B(\xi,(1/2+\delta)\rho_{X(p)})).$$

By Proposition \ref{repulsionbicusp} the balls $ B(\xi,(1/2+\delta)\rho_{X(p)})$ are all disjoint for distinct bicusps $\xi,\xi'\in\CUSP$ except when $\xi,\xi'$ both lie on some $T_k$ with $k=p^{O(\delta)}$. Thus by applying, in order, Proposition \ref{metcom1}, Lemma \ref{heckepullback}, Proposition \ref{repulsioncusp}(c), Lemma \ref{newlemma}, and Proposition \ref{htd}, we have that
\begin{align*}
E &\ll p^{-\delta/2}\vol(C) + p^{-\delta/2}\sum_{k=p^{O(\delta)}} \sum_{\xi\in\CUSP\cap T_k} \vol(C\cap B(\xi,(1/2+\delta)\rho_{X(p)}))\\
&\ll p^{-\delta/2}\vol(C) +  p^{-\delta/2}\sum_{k=p^{O(\delta)}} \sum_{\xi\in\CUSP\cap T_k} \vol'(C\cap B(\xi,(1/2+\delta)\rho_{X(p)}))\\
&\ll p^{-\delta/2}\vol(C) +  p^{-\delta/2}\sum_{k=p^{O(\delta)}} \sum_{\xi\in\CUSP\cap\Delta} \vol'(T_k^*C\cap B(\xi,(1/2+O(\delta))\rho_{X(p)}+O(1)))\\
&\ll p^{-\delta/2}\vol(C) +  p^{O(\delta)}\sum_{k=p^{O(\delta)}} \vol'(T_k^*C\cap B(\Delta,(1/2+O(\delta))\rho_{X(p)}+O(1)))\\
&\ll p^{-\delta/2}\vol(C) +  p^{O(\delta)}\sum_{k=p^{O(\delta)}} \vol(T_k^*C\cap B(\Delta,(1/2+O(\delta))\rho_{X(p)} +O(1)))\\
&\ll p^{-\delta/2}\vol(C) + p^{O(\delta)-1/2}\vol(C).
\end{align*}

Taking $\delta$ sufficiently small establishes the proposition.

\end{proof}

\section{Proof of the Main Theorem}
\label{mainsect}
Recall that for a proper algebraic curve $B$, the gonality $\gon(B)$ of $B$ is the smallest integer $d$ for which there is a degree $d$ map $B\into \P^1$.  For example, $\gon(B)=1$ if and only if $B\cong \P^1$, and $B$ is said to be hyperelliptic if $\gon(B)=2$.  In particular, every genus 1 (or 2) curve is hyperelliptic, but it is easy to show that there are hyperelliptic curves of every genus $g>0$.

In general the gonality of a curve $B$ is difficult to compute, but it is always bounded in terms of the genus $g=g(B)$,
\[\gon(B)\leq g(B)+1\]
by Riemann--Roch.  In fact it is well known from Brill--Noether theory that
\begin{equation}\gon(B)\leq \left\lfloor \frac{g(B)+3}{2}\right\rfloor\label{gon}\end{equation}
is a strict inequality in the sense that a generic curve $B$ will achieve the bound in \eqref{gon}. We now prove the main theorem:

\begin{theorem}  \label{gonality}For any $N>0$, there exists $M_N>0$ such that for any smooth curve $B$ of gonality $n<N$,  any nonconstant map $B\into Z(p)$ factors through a Hecke curve, provided $p>M_N$.
\end{theorem}

\begin{proof}

Suppose not, so that for arbitrarily large $p$ we have a smooth curve $B\into Z(p)$ of bounded gonality $n$ not factoring through a Hecke curve, which we may assume is degree 1 onto its image. The degree $n$ linear system on $B$ gives a map $\phi:\P^1\into \Sym^nZ(p)$ which is also degree 1 onto its image. Let $\psi:C\into (X(p)\times X(p))^n$ be the normalization of an irreducible component of the pullback of $\phi$ to $(X(p)\times X(p))^n$, and $\alpha:C\rightarrow \P^1$ the resulting map:
\[\xymatrixcolsep{3pc}\xymatrix{C\ar[d]_\alpha\ar[r]^-\psi &(X(p)\times X(p))^n\ar[d]\\
\P^1\ar[r]_-\phi&\Sym^nZ(p)}\]

The Galois group $G$ of $(X(p)\times X(p))^n$ over $\Sym^nZ(p)$ is an extension
\[1\into G(p)^n\into G\into S_n\into 1\]
and if $H\subset G$ is the stabilizer of $C$, then $H\backslash C=\P^1$.  Let $C_i=\pi_{i*}C$, and for any $\xi\in X(p)\times X(p)$ by $\mult_\xi (C_i)$ we will mean the multiplicity of $\pi_i\circ\psi:C\into X(p)\times X(p) $ at $\xi$---that is, $\mult_{\xi}(\pi_i\circ\psi (C))$ times the degree of $C$ over its image.  We adopt a similar convention for $C_{ij}=\pi_{ij*}C$ and the multiplicity of $C_{ij}$.

We will bound the degree $\Ram(\alpha)$ of the ramification divisor of $\alpha$.  A point $Q\in C$ ramifies only if $\xi=\psi(Q)$ is in
$$\Delta_g^{ij}: = \{(x_1,y_1,\dots,x_n,y_n,)\mid x_i=gx_j, y_i=gy_j\}\subset (X(p)\times X(p))^n$$
for some $i,j\in\{1,\dots,n\}$ and $g\in G(p)$, or if there exists some $i\in\{1,\dots,n\}$ such that
$\pi_i(\xi)$ is either a singular bicusp or a CM point.  The analytic local stabilizer of $Q$ is a cyclic subgroup of $H$ and therefore of order $O(p)$.  The ramification index of $\alpha$ at $Q$ is then also $O(p)$, and in fact if $\xi$ does not project to a singular bicusp in any projection, the index is $O(1)$ (bounded by $6n!$).  It follows therefore that
 $$\Ram(\alpha)\ll p\sum_i|(\pi_i\circ\psi)^{-1}(\CUSP)|+\sum_i|(\pi_i\circ\psi)^{-1}(\CM)|+\sum_{i,j,g}|(\pi_{ij}\circ\psi)^{-1}(\Delta_g)|.$$

Our multiplicity estimates then give us control over these three terms:
\begin{enumerate}
\item $|\psi^{-1}(\xi)|$ is bounded by the multiplicity $\mult_\xi C$, so
\[|(\pi_i\circ\psi)^{-1}(\CUSP)|\leq\mult_{\CUSP}(C_i).\]
By Proposition \ref{ramificationcusp} we have
\[p|(\pi_i\circ\psi)^{-1}(\CUSP)|=o(\Deg C_i).\]
Note that $\Deg (C_i)=\deg ((\pi_i\circ\psi)^*K_{X(p)\times X(p)})$ accounts for the degree of $C$ over its image under $\pi_i$.

\item Likewise,
\[|(\pi_i\circ\psi)^{-1}(\CM)|\leq \mult_{\CM} (C_i)=o(\Deg C_i)\]
by Proposition \ref{ramification}.  \item Finally, we similarly have
\[\sum_g|(\pi_{ij}\circ\psi)^{-1}(\Delta_g)|\leq \sum_g\mult_{\Delta_g}(C_{ij})=o(\Deg C_{ij})\]
by Proposition \ref{ramificationdiag}.
\end{enumerate}
Thus, $\Ram(\alpha)=o(\Deg C)$, and Riemann--Hurwitz applied to $\alpha$ yields
\[2g(C)-2=o(\Deg C).\]
However, if $d$ is the largest degree of the projections $C\into X(p)$, then as $\Deg(C)=(K\cdot C)$ where $K$ is the canonical divisor of  $(X(p)\times X(p))^n$, we have
\[4\pi nd(2g(X(p))-2)\geq \Deg (C).\]
Riemann--Hurwitz applied to this projection yields
\[2g(C)-2\geq d(2g(X(p))-2)\gg \Deg (C)\]
which is a contradiction.\end{proof}

By the remarks preceding Theorem \ref{gonality}, we obtain as a corollary the following weaker result:
\begin{cor}
\label{genus} For $p>M_N$, every genus $g<N$ curve on $Z(p)$ is a Hecke curve.  
\end{cor}
Corollary \ref{genus} in particular answers a question first posed by Kani and Schanz \cite{kani}:
\begin{cor}\label{rational}For sufficiently large $p$, every rational or elliptic curve in $Z(p)$ is a Hecke curve.
\end{cor}

The surface $Z(p)$ has cyclic quotient singularities, each locally analytically isomorphic to the quotient of $\C^2$ by $\Z/n\Z$ acting by $i\cdot (z,w)=(\zeta^iz,\zeta^{ai}w )$ for a primitive $n$th root of unity $\zeta$ and some $0<a<n$.  For $\CM$ points $n=2$ or $3$, while for singular bicusps $n=p$.  The minimal resolution $q:\tilde Z(p)\into Z(p)$ resolves such a singular point into a chain of smooth rational curves whose intersection form is determined by the continued fraction expansion of $\frac{n}{a}$.

Corollary \ref{rational} also resolves a conjecture of Hermann \cite{hermann}:

\begin{cor}\label{minimal}For all sufficiently large $p$, the minimal model of $Z(p)$ is obtained from $\tilde Z(p)$ by blowing down ``known" curves, \emph{i.e.} by blowing down strict transforms of Hecke curves and curves contracted by $\tilde Z(p)\into Z(p)$.
\end{cor}
\begin{proof}  By Corollary \ref{genus} all rational curves in $\tilde Z(p)$ are of this type.
\end{proof}
\bibliography{biblio}

\begin{thebibliography}{MG78}

\bibitem[Abr96]{Abramovich}
Dan Abramovich.
\newblock A linear lower bound on the gonality of modular curves.
\newblock {\em Internat. Math. Res. Notices}, (20):1005--1011, 1996.

\bibitem[Bea95]{beardon}
A.~F. Beardon.
\newblock {\em The geometry of discrete groups}, volume~91 of {\em Graduate
  Texts in Mathematics}.
\newblock Springer-Verlag, New York, 1995.
\newblock Corrected reprint of the 1983 original.

\bibitem[Bil]{billerey}
N.~Billerey.
\newblock {\em Private communication}.

\bibitem[Bro98]{brooks}
R.~Brooks.
\newblock Platonic surfaces.
\newblock {\em Comm. Math. Helv}, 74:156--170, 1998.

\bibitem[BS94]{BuSa}
P.~Buser and P.~Sarnak.
\newblock On the period matrix of a {R}iemann surface of large genus.
\newblock {\em Invent. Math.}, 117(1):27--56, 1994.
\newblock With an appendix by J. H. Conway and N. J. A. Sloane.

\bibitem[BT13]{BT}
B.~Bakker and J.~Tsimerman.
\newblock On the {F}rey-{M}azur conjecture over low genus curves.
\newblock \href{http://arxiv.org/abs/1309.6568}{\texttt{arXiv:1309.6568}},
  2013.

\bibitem[Car01]{carlton}
David Carlton.
\newblock Moduli for pairs of elliptic curves with isomorphic {$N$}-torsion.
\newblock {\em Manuscripta Math.}, 105(2):201--234, 2001.

\bibitem[Fis11]{fisher}
T.A. Fisher.
\newblock On families of n-congruent elliptic curves.
\newblock \href{http://arxiv.org/abs/1105.1706}{\texttt{arXiv:1105.1706}},
  2011.

\bibitem[Fre97]{frey}
G.~Frey.
\newblock On ternary equations of {F}ermat type and relations with elliptic
  curves.
\newblock In {\em Modular forms and {F}ermat's last theorem ({B}oston, {MA},
  1995)}, pages 527--548. Springer, New York, 1997.

\bibitem[Her91]{hermann}
C.~F. Hermann.
\newblock Modulfl\"achen quadratischer {D}iskriminante.
\newblock {\em Manuscripta Math.}, 72(1):95--110, 1991.

\bibitem[HT02]{hwangto1}
J.~Hwang and W.~To.
\newblock Volumes of complex analytic subvarieties of {H}ermitian symmetric
  spaces.
\newblock {\em American Journal of Mathematics}, 124(6):1221--1246, 2002.

\bibitem[HT12]{hwangto2}
J.~Hwang and W.~To.
\newblock Injectivity radius and gonality of a compact {R}iemann surface.
\newblock {\em American Journal of Mathematics}, 134(1):259--283, 2012.

\bibitem[KS98]{kani}
E.~Kani and W.~Schanz.
\newblock Modular diagonal quotient surfaces.
\newblock {\em Mathematische Zeitschrift}, 227(2):337--366, 1998.

\bibitem[MG78]{mazur}
B.~Mazur and Appendix by~D. Goldfeld.
\newblock Rational isogenies of prime degree.
\newblock {\em Inventiones mathematicae}, 44(2):129--162, 1978.

\bibitem[Zog84]{zograf}
P.~G. Zograf.
\newblock Small eigenvalues of automorphic {L}aplacians in spaces of cusp
  forms.
\newblock {\em Zap. Nauchn. Sem. Leningrad. Otdel. Mat. Inst. Steklov. (LOMI)},
  134:157--168, 1984.
\newblock Automorphic functions and number theory, II.

\end{thebibliography}
\bibliographystyle{alpha}
\end{document}